\documentclass[11pt]{amsart}
\usepackage[T1]{fontenc}
\usepackage{txfonts}
\usepackage{amssymb,verbatim}
\usepackage[all]{xy}
\usepackage{subfig}
\usepackage{tikz}
\usepackage{color}
\usepackage{a4wide}
\usepackage{mathrsfs}
\usepackage{hyperref}
\usepackage{wrapfig}
\usetikzlibrary{arrows}
\DeclareMathAlphabet{\mathpzc}{OT1}{pzc}{m}{it}
\newcommand{\R}{\mathbb{R}}
\newcommand{\C}{\mathbb{C}}

\newcommand\Z{\mathbb{Z}}
\newcommand{\N}{\mathbb{N}}
\newcommand{\Q}{\mathbb{Q}}
\renewcommand{\S}{\mathbb{S}}

\newcommand{\Ab}{\mathbf{A}}

\newcommand{\Bb}{\mathbf{B}}

\newcommand{\Hb}{\mathbf{H}}

\newcommand{\Pb}{\mathbb{P}}

\newcommand{\Tb}{\mathbf{T}}

\newcommand{\Bcal}{\mathcal{B}}
\newcommand{\Ccal}{\mathcal{C}}

\newcommand{\Ecal}{\mathcal{E}}

\newcommand{\Lcal}{\mathcal{L}}
\newcommand{\Mcal}{\mathcal{M}}

\newcommand{\Scal}{\mathcal{S}}
\newcommand{\Tcal}{\mathcal{T}}
\newcommand{\Ucal}{\mathcal{U}}
\newcommand{\Vcal}{\mathcal{V}}
\newcommand{\Wcal}{\mathcal{W}}

\newcommand{\PP}{\mathscr{P}}
\newcommand{\RR}{\mathscr{R}}
\newcommand{\Alp}{\mathcal{A}}
\newcommand{\Aut}{{\rm Aut}}

\newcommand{\card}{{\rm card}}

\newcommand{\id}{\mathrm{id}}
\newcommand{\Id}{\mathrm{Id}}

\newcommand{\Aa}{\textrm{Area}}
\newcommand{\inter}{\mathrm{int}}
\newcommand{\dd}{\mathbf{d}}

\newcommand{\vol}{{\rm vol}}

\newcommand{\ff}{\mathbf{f}}
\newcommand{\pp}{\mathbf{p}}
\newcommand{\qq}{\mathbf{q}}

\newcommand{\rr}{\mathbf{r}}

\newcommand{\proj}{\mathrm{pr}}

\newcommand{\hX}{\hat{X}}
\newcommand{\hM}{\hat{M}}

\newcommand{\homg}{\hat{\omega}}
\newcommand{\hSig}{\hat{\Sigma}}
\newcommand{\hTcal}{\hat{\mathcal{T}}}
\newcommand{\hZ}{\hat{Z}}
\newcommand{\hx}{\hat{x}}

\newcommand{\vide}{\varnothing}
\newcommand{\Sig}{\Sigma}
\newcommand{\sig}{\sigma}

\newcommand{\g}{\gamma}

\newcommand{\del}{\delta}

\newcommand{\veps}{\varepsilon}
\newcommand{\ol}{\overline}
\newcommand{\ul}{\underline}

\newcommand{\ra}{\rightarrow}
\newcommand{\smin}{\setminus}

\newcommand{\strate}{\Omega^d\Mcal_{g,n}(\kappa)}

\newcommand{\pstrate}{\Mcal_{g,n}(\kappa)}

\newcommand{\lstrate}{\Omega\hat{\Mcal}^{\langle \zeta \rangle}_{g,n}(\kappa)}
\newcommand{\ilstrate}{\Omega_1\hat{\Mcal}^{\langle \zeta \rangle}_{g,n}(\kappa)}

\newcommand{\lstrateabdiff}{\Omega\Mcal_{\hat{g}}(\hat{\kappa})}
\newcommand{\strateabdiff}{\Omega\Mcal_{g,n}(\kappa)}

\newtheorem{Theorem}{Theorem}[section]
\newtheorem{Corollary}[Theorem]{Corollary}
\newtheorem{Lemma}[Theorem]{Lemma}
\newtheorem{Proposition}[Theorem]{Proposition}

\newtheorem{Definition}[Theorem]{Definition}

\newtheorem{Claim}[Theorem]{Claim}

\theoremstyle{remark}
\newtheorem{Remark}[Theorem]{Remark}
\begin{document}
\title[Volume forms on moduli spaces of $d$-differentials]{Volume forms on moduli spaces of $d$-differentials}

\author{Duc-Manh Nguyen}

\address{Institut de Math\'ematiques de Bordeaux IMB,\newline
CNRS UMR 5251\newline
Universit\'e de Bordeaux, B\^at. A33\newline
351, cours de la Lib\'eration \newline
33405 Talence Cedex\newline
France}
\email[D.-M.~Nguyen]{duc-manh.nguyen@math.u-bordeaux.fr}

\date{\today}
\begin{abstract}
Given $d\in \N, g\in \N\cup\{0\}$, and an integral vector $\kappa=(k_1,\dots,k_n)$ such that $k_i>-d$ and $k_1+\dots+k_n=d(2g-2)$, let $\strate$ denote the moduli space of meromorphic $d$-differentials on Riemann surfaces of genus $g$ whose zeros and poles have orders prescribed by $\kappa$. We show that $\strate$ carries a canonical volume form that is parallel with respect to its  affine complex manifold (orbifold) structure, and that the total volume of $\Pb\strate=\strate/\C^*$ with respect to the measure induced by this volume form is finite.
\end{abstract}

\maketitle

%\renewcommand{\abstractname}{R\'esum\'e}
%\begin{abstract}
%\end{abstract}
%
%\maketitle

\section{Introduction}
\subsection{Moduli space of $d$-differentials}\label{sec:intro:mod:ps}
Given a compact Riemann surface $X$, we denote by $K_X$ its canonical line bundle.
Let $d$ be a positive integer.
A meromorphic $d$-differential on $X$ is a meromorphic section of $K_X^{\otimes d}$.
For any  $g \in \Z_{\geq 0}$, and any integral vector $\kappa=(k_1,\dots,k_n)$  such that $k_i>-d$ and $k_1+\dots+k_n=d(2g-2)$, let $\strate$ denote the space of pairs $(X,q)$, where $X$ is a compact Riemann surface of genus $g$, and $q$ is a non-zero meromorphic $d$-differential on $X$ whose zeros and poles have orders prescribed by $(k_1,\dots,k_n)$. In particular, $q$ has exactly $n$ zeros and poles, and all the poles of $q$ have order at most $d-1$.

The space $\strate$ is called a {\em stratum} of  $d$-differentials in genus $g$. Each stratum may have several components. To lighten the notation, throughout this paper, by $\strate$ we will mean a connected component of the corresponding stratum.

There is a natural $\C^*$-action on $\strate$ by multiplying to the $d$-differential a non-zero scalar. We will denote by $\Pb\strate$ the projectivization of $\strate$, that is $\Pb\strate=\strate/\C^*$.

\medskip

The strata $\Omega^1\Mcal_{g,n}(\kappa)$ and $\Omega^2\Mcal_{g,n}(\kappa)$ are involved in various domains like dynamics in Teichm\"uller spaces, interval exchange transformations, billiards in rational polygons, and have now become classical objects of study. Recently, the connections between the moduli spaces of $d$-differentials, for $d \in \{2,3,4,6\}$, with tiling problems on surfaces has been brought to light in the works~\cite{ES18, Engel-I}. More generally, $\Pb\strate$ are of interest since they arise as natural subvarieties of $\Mcal_{g,n}$ (see for instance \cite{FP18, Sch16}). Holomorphic $d$-differentials also appear in the study of Hitchin representations of fundamental group of surfaces \cite{Lab17}.

It is a well known fact that $\Omega^1\Mcal_{g,n}(\kappa)$ and $\Omega^2\Mcal_{g,n}(\kappa)$ are algebraic orbifolds, which admit an affine complex orbifold structure (there is an atlas  whose charts map open subsets onto finite quotients of open subsets in a complex vector space, and transition maps are given by complex linear maps).
Moreover $\Omega^1\Mcal_{g,n}(\kappa)$ and $\Omega^2\Mcal_{g,n}(\kappa)$ carry a natural volume form $d\vol^*$, which is called the {\em Masur-Veech volume}.
By definition, $d\vol^*$ is parallel with respect to the affine orbifold structure, that is in a local chart of the affine structure $d\vol^*$ differs from the Lebesgue measure by a constant.

The Masur-Veech volume induces a volume form $d\vol^*_1$ on $\Pb\Omega^1\Mcal_{g,n}(\kappa)$ and $\Pb\Omega^2\Mcal_{g,n}(\kappa)$
as follows: denote by $\Omega_1^1\Mcal_{g,n}(\kappa)$ (resp. $\Omega_1^2\Mcal_{g,n}(\kappa)$) the set of $(X,q)\in \Omega^1\Mcal_{g,n}(\kappa)$ (resp. $(X,q)\in \Omega^2\Mcal_{g,n}(\kappa)$) such that the area of $X$ with respect to the flat metric defined by $q$ is at most $1$. 
Then $d\vol^*_1$ is the pushforward of the restriction of $d\vol^*$ to $\Omega^1_1\Mcal_{g,n}(\kappa)$ and $\Omega^2_1\Mcal_{g,n}(\kappa)$ to $\Pb \Omega^1\Mcal_{g,n}(\kappa)$ and $\Pb\Omega^2\Mcal_{g,n}(\kappa)$ respectively.
It is a classical result due to Masur and Veech that the total volumes of $\Pb\Omega^1\Mcal_{g,n}(\kappa)$ and $\Pb\Omega^2\Mcal_{g,n}(\kappa)$ with respect to $d\vol^*_1$ are finite. This result is fundamental for the study of Teichm\"uller dynamics on moduli spaces of Abelian and quadratic differentials.

\medskip

For  $d\geq 3$,  the spaces $\strate$ have been studied by several authors from both flat metric and complex algebraic geometry points of view (see for  instance \cite{Veech:FS,Th98,McM_cyclic,BCGGM2,Sch16,GP17}).  In particular, it has been shown that each stratum $\strate$ is a complex orbifold, and if there exists  $(X,q) \in \strate$   such that $q$ is \ul{not} a $d$-th power of an Abelian differential on $X$ then (see \cite{Veech:FS,BCGGM2,Sch16})
$$
\dim_\C\strate=2g+n-2.
$$
Moreover,  $\strate$ also has a complex affine orbifold structure which is defined in the same manner as $\Omega^2\Mcal_{g,n}(\kappa)$.
%Note that the dimension of $\strate$ for $d=1,2$ has been known since the work\cite{HM79} of Hubbard and Masur (see also \cite[Prop. 1.2]{MS91} and \cite{Veech86}).

In this paper, we  extend the result of Masur and Veech to the case $d \geq 3$. Our main result is the following

\begin{Theorem}~\label{thm:main}
For all $d\in \N$, $\strate$ carries a canonical parallel volume form $d\vol$ with respect to its affine  structure. 
Let $d\vol_1$ be the induced  volume form on $\Pb\strate$.
Then  the total volume of $\Pb\strate$ with respect to $d\vol_1$ is finite.
\end{Theorem}

\begin{Remark}\label{rmk:main:thm}\hfill
\begin{itemize}
\item[(i)] In the case $d\in \{1,2\}$, the existence of a parallel volume form on $\strate$ is deduced from the fact that the transition maps of the affine structure preserve a lattice in $\C^N$. This phenomenon also occurs in the case $d\in \{3,4,6\}$. However for general $d$ such a lattice does not exist, hence the existence of $d\vol$ is not obvious.

 \item[(ii)] In the construction of $d\vol$, one needs  to fix a primitive $d$-th root $\zeta$ of unity. The apparent dependence on the choice of $\zeta$ can be easily removed by a correction factor (see Proposition~\ref{prop:vol:form:def}).

 \item[(iii)] A different approach to define volume forms on moduli spaces of flat surfaces in general can be found in \cite{Veech:FS}. In a sense, our definition of $d\vol$ is an adaptation of this construction in the case of flat surfaces defined by $d$-differentials.

 \item[(iv)] By construction, the volume form $d\vol$ agrees, up to a multiplicative constant, with the Masur-Veech measure on $\strate$ when $d=1$. In Section~\ref{sec:vol:form:compare}, we compute this constant explicitly (see Proposition~\ref{prop:compare:v:form:d:1}).

 \item[(v)]   In~\cite{Engel-I}, Engel introduces other volume forms on $\strate$ which generalize the Masur-Veech measure for $d\in \{3,4,6\}$. Those volume forms arise naturally in the counting of tilings of surfaces by triangles, squares, or hexagons. We will show that those volume forms always differ from $d\vol$ by a constant in $\Q\cup\sqrt{3}\Q$ (see Proposition~\ref{prop:vol:const:rat}).

 \item[(vi)] In the case $g=0$, the finiteness of the volume of $\Pb\strate$ (with respect to some volume forms  equivalent to $d\vol_1$)  has been proved in \cite{Veech:FS} and \cite{Th98} (see also \cite{Ng1,Ng2}).
\end{itemize}
\end{Remark}

\subsection{Outline}\label{sec:outline}
\subsubsection{Definition of the volume form}\label{subsec:outline:def:v:form}
We will now give a brief description of the volume forms $d\vol$ and $d\vol_1$, and refer to Section~\ref{sec:vol:form} for more details.
%Let us fix a compact, closed, oriented  topological surface $M$ of genus $g$, and a finite subset $\Sig=\{s_1,\dots,s_n\}$ of  $M$. Let $M':=M\setminus \Sig$.

Let $(X,q)$ be an element of $\strate$. Let $Z:=\{x_1,\dots,x_n\}$ be  the set of zeros and poles of $q$, where the order of $x_i$ is $k_i$.
Throughout  this paper, the points in $Z$ are always supposed to be numbered.

Pick a primitive $d$-th root $\zeta$ of the unity.
Associated to the pair $(X,q)$ we have a triple $(\hX,\homg,\tau)$, where $\hX$ is a cyclic covering of degree $d$ over $X$, $\homg$ is an Abelian differential on $\hX$ such that the pullback of $q$ to $\hX$ is equal to $\homg^d$, and $\tau$ is an automorphism of order $d$ of $\hX$ such that $X \simeq \hX/\langle \tau \rangle$ and $\tau^*\homg=\zeta\homg$.  The triple $(\hX,\homg,\tau)$ will be called the {\em canonical cyclic cover} of $(X,q)$.

Denote by $\hZ$ the inverse image of $Z$ in $\hX$.
Consider $V_\zeta:=\ker(\tau-\zeta\id) \subset  H^1(\hX,\hZ,\C)$.
A neighborhood of $(X,q)$ in $\strate$ can be identified with the quotient of an open subset of $V_\zeta$ by $\Aut(X,q)$ (see Section\ref{sec:loc:coord:per}). In particular, we have
$$
\dim \strate=\dim V_\zeta=N.
$$
Let $\pp: H^1(\hX,\hZ,\C) \ra H^1(\hX,\C)$ be the natural projection.
Let $Z_0$ be the subset of $Z$ consisting of $x_i\in Z$ such that $k_i\in d\Z$.
Then $\dim \ker\pp\cap V_\zeta=\card(Z_0)$ if $d\geq 2$ (see Proposition~\ref{prop:ker:project:coh}).
Without loss of generality, let us assume that $Z_0=\{x_1,\dots,x_r\}$.

For $i=1,\dots,r$, pick a point $\hx_i$ in the inverse image of $\{x_i\}$ in $\hX$.
Let $c_i$ be a path from $\hx_i$ to $\tau(\hx_i)$.
We consider $c_i$ as an element of $H_1(\hX,\hZ,\Z)\subset H^1(\hX,\hZ,\C)^*$.
Denote by $c_i^\zeta$  the restriction of $c_i$ to the subspace $V_\zeta$.

The intersection form on $H_1(\hX,\Z)$ defines a Hermitian product $(.,.)$ on $H^1(\hX,\C)$.
Let $\vartheta$ denote the imaginary part of $(.,.)$.
Define
\begin{equation}\label{eq:intro:def:Theta}
\Theta:=\frac{1}{|1-\zeta|^{2r}}\cdot\frac{1}{(N-r)!}\left(\frac{\imath}{2}\right)^r \pp_{|V_\zeta}^*\vartheta^{N-r}\wedge c_1^\zeta\wedge\bar{c}_1^\zeta\wedge\dots\wedge c_r^\zeta\wedge \bar{c}_r^\zeta \in \Lambda^{N,N}V^*_\zeta.
\end{equation}
Then $\Theta$ gives a well defined volume form $d\vol$ on $\strate$.  Note that the factor $1/|1-\zeta|^{2r}$ is introduced so that eventually $\Theta$ does not depend on the choice of $\zeta$ (see Lemma~\ref{lm:vol:form:Vzeta} and Proposition~\ref{prop:vol:form:def}).

\begin{Remark}\label{rk:alter:def:vol:intro}
It is worth noticing that in \eqref{eq:intro:def:Theta}, $\Theta$ is defined as the restriction of an $(N,N)$-form on $H^1(\hX,\hZ,\C)$ to $V_\zeta$.
An alternative method to define the volume form $d\vol$ is as follows: consider the following exact sequence of cohomology with coefficients in $\C$
\begin{equation*}
0 \to H^0(\hX)\to H^0(\hZ) \to H^1(\hX,\hZ) \to H^1(\hX) \to 0.
\end{equation*}
The automorphism $\tau$ acts equivariantly on those spaces.
Let $H$ stand for one of the cohomology spaces in the exact sequence above.
Denote by $H_\zeta$ the $\zeta$-eigenspace of the action of $\tau$ on $H$.
Since $\zeta\neq 1$, we have $H^0(\hX)_\zeta=\{0\}$, and the exact sequence above induces the following
\begin{equation*}
0 \to  H^0(\hZ)_\zeta \to H^1(\hX,\hZ)_\zeta \to H^1(\hX)_\zeta \to 0.
\end{equation*}
Thus, to define a volume form on $V_\zeta =H^1(\hX,\hZ)_\zeta$, it is enough to give a volume form on $H^0(\hZ)_\zeta$ and a volume form on $H^1(\hX)_\zeta$ (see Remark~\ref{rk:def:vol:ex:seq} for more details).
\end{Remark}

By a slight abuse of notation, let us denote by $(.,.)$ the pullback of the intersection form on $H^1(\hX,\C)$ to $H^1(\hX,\hZ,\C)$ by $\pp$.
Let $\Pb V_\zeta$ denote the projective space of $V_\zeta$,
and $\Pb V_\zeta^+ \subset \Pb V_\zeta$ be the set of lines $\C\cdot v$, with $v\in V_\zeta$ such that $(v,v) >0$.
We get a volume form $\Theta_1$ on $\Pb V_\zeta^+$ from $\Theta$ as follows:
given a subset $U$ of $\Pb V_\zeta^+$, let $C(U)$ denote the cone over $U$, that is $C(U)=\{v\in V_\zeta, \; \C\cdot v \in U\}$. Let $C_1(U)$ denote the intersection of $C(U)$ with the set $\{v \in V_\zeta, 0 < (v,v) \leq 1\}$. Then $\Theta_1(U)$ is defined to be $\Theta(C_1(U))$.
It is not difficult to check that $\Theta_1$  gives a well defined volume form $d\vol_1$ on $\Pb\strate$.

\subsubsection{Finiteness of the total volume}\label{subsec:outline:finite:vol}
To prove that the volume of $\Pb\strate$ with respect to $d\vol_1$ is finite, we will consider the space $\lstrate$ of triples $(\hX,\homg,\tau)$ that are canonical cyclic covers of the elements of $\strate$.
By construction, there is a bijection $\PP: \lstrate \ra \strate$ which sends $(\hat{X},\hat{\omega},\tau)$ to $(X,q)$ (see Lemma~\ref{lm:ab:diff:eig:vect}).
Locally, one can identify open subsets of $\lstrate$ with open subsets of $V_\zeta$.
Hence $\Theta$  gives rise to a well defined volume form on $\lstrate$  which is also denoted by $d\vol$.
Set
$$
\ilstrate:=\{(\hat{X},\hat{\omega},\tau) \in \strate, \; \Aa(\hat{X},\hat{\omega}) \leq 1\}.
$$
By definition, we have $\vol_1(\Pb\strate)=\vol(\ilstrate)$.
We will show that $\vol(\ilstrate)$ is finite.
To this end, we  make use of Delaunay triangulations of flat surfaces, and follow a strategy similar to the one in \cite{MS91}.

Given $(\hX,\homg, \tau) \in \lstrate$, we  equip $\hX$ with the flat metric defined by $\homg$. 
Denote by $\hZ$ the zero set of $\homg$.
By definition $\tau$ corresponds to an isometry of order $d$ of $\hX$ preserving $\hZ$.
Call a cylinder on $\hX$ a {\em long cylinder} if its height is greater than a universal constant times the square-root of the area of $\hX$ (see Def.~\ref{def:long:cyl}).
It is not difficult to see that long cylinders are pairwise disjoint.

Let $\hTcal$ be a Delaunay triangulation of the pair $(\hX,\hZ)$ invariant by $\tau$.
The edges of $\hTcal$  have length  bounded by the square-root of the area of $\hX$, except those that cross some long cylinders.
Moreover, each edge of $\hTcal$ can cross at most one long cylinder, and can not cross the same long cylinder twice (see Prop.~\ref{prop:long:sc:high:cyl} and Lem.~\ref{lm:sc:cross:long:cyl}).
It follows that the set of edges that cross a fixed long cylinder corresponds to a {\em simple cycle} in the dual graph of $\hTcal$, that is the image of an injective continuous map from $\S^1$ to the dual graph.
We abusively call a set of edges in $\hTcal$ a {\em simple cycle} if this family is dual to a simple cycle in the dual graph.
Each long cylinder (if exists) corresponds uniquely to a simple cycle in $\hTcal$, and the simple cycles corresponding to two distinct long cylinders are disjoint.
Call a collection $\tilde{\gamma}$ simple cycles in $\hTcal$ {\em admissible} if
\begin{itemize}
 \item this collection  is $\tau$ invariant,

 \item the simple cycles in $\tilde{\gamma}$  are pairwise disjoint.
\end{itemize}
Let $N_1$ be the number of geometric edges of $\hTcal$ (recall that a geometric edge of a graph corresponds to a pair of inversely directed edges).
Note that   $N_1$  is completely determined by the genus of $\hX$ and the cardinality of $\hZ$.
We identify $\C^{2N_1}$ with the space of complex valued functions on the set of directed edges of $\hTcal$ .
To each pair $(\hTcal,\tau)$ together with an admissible collection  $\tilde{\gamma}$ of simple cycles, we specify an open subset $U$ in a linear subspace $W\subset \C^{2N_1}$, and define a locally homeomorphic map $\Psi$ from $U$ to $\lstrate$.
We also specify an open subset $U^1$ of $U$ such that  $\Psi^{-1}(\ilstrate) \subset U^1$.
Because the set of pairs $(\hTcal,\tau)$ is finite, and given such a pair the set of admissible collections of simple cycles is finite, it follows that the map $\Psi$ belongs to a finite set.

Every flat surface in $\lstrate$ admits a Delaunay triangulation  invariant by $\tau$ (see Prop.~\ref{prop:inv:Delaunay:tria}). Therefore, $\lstrate$ is covered by the images of the maps $\Psi$ as above.
It follows that $\ilstrate$ is covered by the finite  family  of open subsets $\{\Psi(U^1)\}$ of $\lstrate$.
Since the volume form $\vol$ on $\lstrate$ differs from the Lebesgue measure on $U^1$ by a constant, the finiteness of $\vol(\ilstrate)$ follows from the fact that $U^1$ has finite Lebesgue volume.

Recall that we identify $\C^{2N_1}$ with the space of complex valued functions on the set of directed edges of $\hTcal$.
To show that the volume of $U^1$ is finite, we choose an appropriate family of $N$ directed edges $\{e_1,\dots,e_N\}$ in $\hTcal$ such that the map from $\C^{2N_1}$ to $\C^N$ sending $z\in \C^{2N_1}$ to $(z(e_1),\dots,z(e_N))\in \C^N$ induces an isomorphism from $W$ onto $\C^N$.
We then split the set $\{e_1,\dots,e_N\}$ into two subsets: $\{e_1,\dots,e_k\}$ is the set of edges that are contained in one of the simple cycles in the collection $\tilde{\gamma}$, they correspond to edges that cross some long cylinders, and $\{e_{k+1},\dots,e_N\}$ are the remaining edges.
By  properties of Delaunay triangulations, $(z(e_{k+1}),\dots,z(e_N))$ is contained in a ball (of finite radius) in $\C^{N-k}$, and once $(z(e_{k+1}),\dots,z(e_N))$ is fixed, each of $z(e_1),\dots,z(e_k)$ is contained in a rectangle whose area is uniformly bounded.
Using Fubini theorem, we conclude that $U^1$ has finite volume, and Theorem~\ref{thm:main} follows.

\subsection*{Organization} The paper is organized as follows, in Section~\ref{sec:top:prelim} and Section~\ref{sec:loc:coord:per} we recall the classical constructions of the period mappings which define  the complex affine orbifold structure on $\strate$ (see Proposition~\ref{prop:eigen:sp:loc:coord}). The results in these sections are not new, we include their proof for the completeness, and more importantly, to settle the framework of the subsequent discussion. In Section~\ref{sec:proj:abs:coh}, we study the restriction of the projection $\pp: H^1(\hat{M},\hat{\Sig},\C) \ra H^1(\hat{M},\C)$ to $V_\zeta$. In particular, we determine $\ker(\pp_{|V_\zeta})$ and $\mathrm{Im}(\pp_{|V_\zeta})$. In Section~\ref{sec:vol:form}, we show that $\Theta$ is a volume form on $V_\zeta$, which gives rise to a well defined volume form $d\vol$ on stratum proving the first part of Theorem~\ref{thm:main}.
In Section~\ref{sec:Delaunay}, we recall basic properties of the Delaunay triangulations on flat surfaces, which will be used in the proof of the second part of Theorem~\ref{thm:main}. Finally, in Section~\ref{sec:proof:finite} we will give the proof that $\vol_1(\Pb\lstrate)$ is finite.

\subsection*{Notation and convention} Throughout this paper, $M$ will be an oriented, compact, closed, connected surface of genus $g$, $\Sig=\{s_1,\dots,s_n\}$ is a finite subset of cardinality $n$ of $M$, $M'=M\setminus\Sig$. We will always suppose that
\begin{equation}\label{eq:Euler:char:neg}
\chi(M')=2-2g-n <0.
\end{equation}
Given $(X,q) \in \strate$, the zeros and poles of $q$ are always supposed to be numbered. This numbering is specified by a homeomorphism between the pairs $(M,\Sig)$ and $(X,Z(q))$.
Some of the numbers $k_1,\dots,k_n$ may be $0$, in which case the corresponding points are marked points on the surface $X$ that are neither zero nor pole of $q$. In particular, they can be chosen arbitrarily on the surface $X$.

The dimension of $\strate$ is denoted by $N$, and $\zeta$ is a fixed primitive $d$-th root of unity.

%\subsection*{Acknowledgement} The author is grateful to the anonymous referee for the helpful comments, in particular on the construction of the volume form,  which help to improve the exposition of the paper.
%*******************************************
%*******************************************
%*******************************************
%*******************************************

\section{Topological preliminaries}\label{sec:top:prelim}
\subsection{Cyclic covering}\label{sec:cyclic:cov}
We will now introduce some topological constructions which will allow us to define local charts for $\strate$.
Pick a base point $s_0\in M'$.  Let $\{\alpha_1,\beta_1,\dots,\alpha_g,\beta_g\}$  be a standard generating set of $\pi_1(M,s_0)$, that is
\begin{equation}\label{eq:gen:set:pi1:M}
\pi_1(M,s_0) \simeq \langle \alpha_1,\beta_1,\dots,\alpha_g,\beta_g \; \big| \prod_{i=1}^g [\alpha_i,\beta_i]=1\rangle.
\end{equation}
Note that  all the $\alpha_i,\beta_i$ can be represented by loops missing the set $\Sig$. For $i=1,\dots,n$, let $\del_i$ be an element of $\pi_1(M',s_0)$ represented by a loop that is freely homotopic to the boundary of a small disc about $s_i$ such that the following holds

\begin{equation}\label{eq:gen:set:pi1:Mp}
\pi_1(M',s_0) \simeq \langle \alpha_1,\beta_1,\dots,\alpha_g,\beta_g,\del_1,\dots,\del_n \; \big| \prod_{i=1}^g[\alpha_i,\beta_i]\cdot\prod_{j=1}^n \del_j=1\rangle.
\end{equation}

In what follows, for any $d \in \N$, we will identify the group $\Z/d\Z$ with $\{e^{\frac{2\pi\imath}{d}k}, \; k=0,\dots,d-1\}$ using the identification $k \mapsto e^{\frac{2\pi\imath}{d}k}$.
Let $\veps: \pi_1(M',s_0) \ra \Z/d\Z$ be a group morphism.
Since $\Z/d\Z$ is Abelian, \eqref{eq:gen:set:pi1:Mp} implies
\begin{equation}\label{eq:cond:hom:at:sing}
\prod_{j=1}^n \veps(\del_j)=1.
\end{equation}

Assume that $\veps$ is surjective.
Then $\Gamma=\ker\veps$ is a normal subgroup of index $d$ in $\pi_1(M',s_0)$.
Let $\widetilde{M'}$ be the universal cover of $M'$, and $\hat{M}':=\widetilde{M'}/\Gamma$.
By construction, there is a covering map $\pi: \hat{M}' \ra M'$ of degree $d$.

Since  $\hat{M}'$ is of finite type, it is homeomorphic to a punctured surface $\hat{M}\setminus\hat{\Sig}$, where $\hat{M}$ is a compact closed surface, and $\hat{\Sig}$ is a finite subset of $\hat{M}$. The covering map $\pi$ extends to a (continuous) map $\pi: \hat{M} \ra M$ such  that $\pi^{-1}(\Sig)=\hat{\Sig}$. For any $i\in \{1,\dots,n\}$, the cardinality of  $\pi^{-1}(\{s_i\})$ can be computed as follows: let $d_i$ be the order of $\veps(\del_i)$ in $\Z/d\Z$, then $\card(\pi^{-1}(\{s_i\}))=\frac{d}{d_i}$. The genus $\hat{g}$ of $\hat{M}$ can be computed by the Riemann-Hurwitz formula. Namely,
\begin{equation}\label{eq:cov:gen:R-H}
2\hat{g}-2=d(2g-2)+\sum_{i=1}^n\frac{d}{d_i}(d_i-1) \Leftrightarrow \hat{g}=d(g-1)+1+\frac{1}{2}\left(nd-\sum_{i=1}^n \frac{d}{d_i}\right).
\end{equation}

%************ unnecessary remark******************************
% \begin{Remark}\label{rk:compute:di}
% Let $k_i=\veps(\delta_i) \in \Z/d\Z$. Let $\ell_i=\gcd(n,k_i)$. Then the order of $k_i$ in $\Z/d\Z$ is given by $d_i=\frac{n}{\ell_i}$. Equivalently, we have $\ell_i=\frac{d}{d_i}$.
%
% Note that the condition \eqref{eq:cond:hom:at:sing} implies that $\sum_{i=1}^nk_i \equiv 0 \mod d$. If $d$ is even then $\#\{i \in \{1,\dots,n\}, \; k_i \text{ is odd}\}$ is even. Since in this case $\ell_i$ is even if and only if $k_i$ is even, it follows that $\sum_{i=1}^n\ell_i=\sum_{i=1}^n \frac{d}{d_i}$ is even. Thus
% $$
% \frac{1}{2}\sum_{i=1}^d\frac{d}{d_i}(d_i-1)=\frac{1}{2}\left(nd-\sum_{i=1}^n \frac{d}{d_i}\right) \in \N.
% $$
% If $d$ is odd, then $d_i$ is also odd, and $d_i-1$ is even. Therefore, we also have
% $$
% \frac{1}{2}\sum_{i=1}^d\frac{d}{d_i}(d_i-1)\in \N.
% $$
% \end{Remark}
%****************************************************************

Pick a base point $\hat{s}_0$ in $\hat{M}'$ such that $\pi(\hat{s}_0)=s_0$. Since the covering $\pi: \hat{M}' \ra M'$ is Galoisian, any element $\alpha \in \pi_1(M',s_0)$ lifts to an automorphism of $\pi$, that is a homeomorphism $T_\alpha: \hat{M}' \ra \hat{M}'$ such that $\pi=T_\alpha\circ\pi$. By construction, the loops that represent $\alpha$ lift to  paths from $\hat{s}_0$ to $T_\alpha(\hat{s}_0)$ in $\hat{M}'$.
Note that $T_\alpha$ can also be defined as the action of $\alpha$ on the quotient $\widetilde{M'}/\Gamma$.
It is not difficult to check that the map $T_\alpha$ extends to a homeomorphism from $\hat{M}$ to itself preserving the set $\hat{\Sig}$.

If $\alpha'$ is another element of $\pi_1(M',s_0)$ such that $\veps(\alpha)=\veps(\alpha')$, then  $T_{\alpha'}=T_\alpha$, since $\alpha'\cdot\alpha^{-1}\in \Gamma$. Thus the homeomorphism $T_\alpha$ depends only on $\veps(a) \in \Z/d\Z$. In what follows, given $\zeta \in \Z/d\Z$, we will denote the homeomorphism of $\hat{M}$  associated with $\zeta$ by $T_\zeta$. By construction, we have $T^d_\zeta=\id_{\hat{M}}$, and $T^k_\zeta \neq \id_{\hat{M}}$ for all $k\in \{1,\dots,d-1\}$.

\subsection{Coverings associated with $d$-differentials}\label{sec:d:diff:n:cover}
Let $(X_0,q_0)$  be an element of  $\strate$, and denote by $Z(q_0)=\{x^0_1,\dots,x^0_n\}$ the set of zeros and poles of $q_0$, where the order of $x^0_i$ is $k_i$.
Let $\veps_0: \pi_1(X_0\smin Z(q_0),*) \ra \Z/d\Z$ be the group morphism given by the linear holomomies of the flat metric defined by $q_0$.

\begin{Lemma}\label{lm:prim:d:diff}
 The morphism $\veps_0$ is surjective if and only if $q_0$ is not a power of a $d'$-differential on $X_0$ with $d' <d$.
\end{Lemma}
\begin{proof}
 If $q_0=q_1^k$, where $q_1$ is a meromorphic $d/k$-differential on $X$, then $\veps_0$ would take values in $\Z/(d/k)\Z=\{e^{\frac{2\pi\imath k}{d}j}, \; j=0,1,\dots,d/k-1\}$.  Conversely, if $\veps_0$ is not surjective, then it would take values in a proper subgroup of $\Z/d\Z$. Hence $\mathrm{Im}(\veps_0)=\Z/d'\Z$, for some $d' < d$ that divides $d$.  It follows that the $d'$-differentials $(dz)^{d'}$ in the local charts associated with the flat metric defined by $q_0$ match together to give  a well defined $d'$-differential $q_1$ on $X$. By construction, we have $q_0=q_1^{d/d'}$ (since both $q_0$ and $q_1^{d/d'}$ are given by $(dz)^d$ on the local charts of the flat metric).
\end{proof}

\begin{Remark}\label{rk:prim:comp}
Since $\veps_0$ takes values in  a discrete set, if it is surjective for $(X_0,q_0)$, then it is surjective for all $(X,q)$ in the same connected component of $\strate$.
\end{Remark}

\begin{Definition}\label{def:prim:d:diff}
 A $d$-differential is said to be {\em primitive} if it is \ul{not} a power of some $d'$-differential, with $d' < d$, on the same Riemann surface. A component of $\strate$ is said to be {\em primitive} if it contains  a primitive $d$-differential (hence all of its elements are primitive).
\end{Definition}
From now on we will suppose that $(X_0,q_0)$ is primitive, and to simplify the notation, we will write $Z_0$ instead of $Z(q_0)$.

\medskip

Fix a homeomorphism $f_0:M \ra X_0$ such that $f_0(s_i)=x^0_i$. In particular, we have $f_0(M')=X_0\setminus Z_0$.
Let $\Gamma=\ker(\veps_0\circ f_{0*})\subset \pi_1(M',*)$, and $\pi: \hat{M}' \ra M'$ be the covering associated with $\Gamma$.
Let $\hat{M}$ and $\hat{\Sig}$  be as in Section~\ref{sec:cyclic:cov}. Fix a generator $\zeta$ of $\Z/d\Z$, and let $T_\zeta: \hat{M} \ra \hat{M}$ be the homeomorphism of $\hat{M}$ associated with $\zeta$.
Recall that we have $T_\zeta(\hat{\Sig})=\hat{\Sig}$ and $T^d_\zeta=\id_{\hat{M}}$.

The following lemma is classical (see \cite{EV92,BCGGM2} for a proof by complex algebraic geometry arguments).
\begin{Lemma}\label{lm:ab:diff:eig:vect}
Given $(X_0,q_0)\in \strate$, there is a triple $(\hX_0,\homg_0,\tau_0)$, where $\hX_0$ is a compact Riemann surface, $\homg_0$ a holomorphic $1$-form on $\hX_0$, and $\tau_0$ an automorphism of order $d$ of $\hX_0$ such that
\begin{itemize}
\item[$\bullet$] $X_0 \simeq \hX_0/\langle \tau_0\rangle$,

\item[$\bullet$] the pullback of $q_0$ to $\hX_0$ is equal to $\homg_0^d$,

\item[$\bullet$] $\tau_0^*\homg_0=\zeta\homg_0$.
\end{itemize}
Moreover, there is a homeomorphic map $\hat{f}_0: \hat{M} \ra \hat{X}_0$ such that
\begin{itemize}
\item[(i)] $f_0\circ\pi=\varpi_0\circ \hat{f}_0$, where $\varpi_0: \hX_0\to X_0$ is the canonical projection, and

\item[(ii)] $T=\hat{f}^{-1}_0\circ\tau_0\circ\hat{f}_0$.
\end{itemize}
The triple $(\hX_0,\homg_0,\tau_0)$ will be called the {\em canonical cyclic cover} of $(X_0,q_0)$.
\end{Lemma}
\begin{proof}[Sketch of proof]
Let $X'_0$ denote the punctured surface $X_0\smin Z_0$.  Let $\tilde{\varpi}'_0: \Delta \ra X'_0$ be the universal covering map, where $\Delta=\{z \in  \C, \; |z| <1\}$.
The $d$-differential $(\tilde{\varpi}'_0)^*q_0$ admits a well defined $d$-th root $\xi_0(z)dz$ on $\Delta$ which is nowhere vanishing.
The $1$-form $\xi_0(z)dz$ descends to a well defined holomorphic $1$-form $\hat{\omega}'_0$ on  the quotient $\hat{X}'_0:= \Delta/\ker(\veps_0)$.

By construction, we have a covering map $\varpi': \hat{X}'_0 \ra X'_0$ of degree $d$.
Therefore $\hat{X}'_0$ is a Riemann surface with punctures, that is there is a compact Riemann surface $\hat{X}_0$ and a finite subset $\hat{Z}_0$ of $\hat{X}_0$ such that $\hat{X}'=\hat{X}_0\smin\hat{Z}_0$.
The covering map $\varpi'$ extends to a branched covering $\varpi: \hat{X}_0 \ra X_0$ with branched points in $\hat{Z}_0$.

Since the poles of $q_0$ have order at most $d-1$, $|(\varpi')^*q_0|$ is bounded in a neighborhood of any $\hat{x} \in \hat{Z}_0$.
Since $(\hat{\omega}'_0)^d=\varpi^*q_0$ on $\hat{X}'_0$, it follows that $|\hat{\omega}'_0|$ is bounded in a neighborhood of any puncture of $\hat{X}'$.
Thus $\hat{\omega}'_0$ extends to a holomorphic $1$-form $\hat{\omega}_0$ on $\hat{X}_0$ which does not vanish on $\hat{X}'_0$.
By construction, we have $\hat{\omega}_0^d=\varpi^*q_0$ on $\hat{X}_0$.

Let $\alpha$ be an element of $\pi_1(X'_0,*)$ such that $\veps_0(\alpha)=\zeta$. The action of $\alpha$ on $\Delta$ induces an automorphism $\tau'_0$ on $\hat{X}'_0$ which satisfies  $(\tau'_0)^*\hat{\omega}'_0=\zeta\hat{\omega}'_0$. One can readily check that $\tau_0$ extends to an automorphism $\tau_0$ of $\hat{X}_0$ of order $d$.

There is a homeomorphism $\hat{f}'_0: \hat{M}'\ra \hat{X}'_0$ such that $\varpi'\circ\hat{f}'_0=f_0\circ\pi$.
This homeomorphism extends uniquely to a homeomorphism $\hat{f}^0: \hat{M} \to \hat{X}_0$ such that $\hat{f}_0(\hat{\Sig})=\hat{Z}_0$.
It is  straightforward to check that $(\hat{X}_0, \hat{\omega}_0,\tau_0)$, and $\hat{f}_0$ satisfy all the required properties.
\end{proof}

\begin{Remark}\hfill
\begin{itemize}
  \item[(a)] The form $\hat{\omega}_0$ in Lemma~\ref{lm:ab:diff:eig:vect} is obviously not unique if $d>1$, since multiplying by a $d$-th root of unity provides us with another holomorphic $1$-form with the same properties.

  \item[(b)] Some of the points in $\hZ_0$ may not be zero of $\hat{\omega}_0$, those points  are just marked points on $\hat{X}_0$. However,  to lighten the discussion we will call $\hZ_0$ the zero set of $\hat{\omega}_0$.

  \item[(c)] The map $\hat{f}_0: \hM_0 \to \hX_0$  induces an isomorphism $\hat{f}_{0,\zeta}: H^1(\hM,\hSig,\C)_\zeta \to H^1(\hX_0,\hZ_0,\C)_\zeta$, where $H^1(\hM,\hSig,\C)_\zeta$ and  $H^1(\hX_0,\hZ_0,\C)_\zeta$ are the eigenspaces of the eigenvalue $\zeta$ of the actions of $T$ and $\tau_0$ on $H^1(\hM,\hSig,\C)$ and $H^1(\hX_0,\hZ_0,\C)$ respectively.
Note that $\hat{f}_0$  (more precisely, the homotopy class of $\hat{f}_0$) is not unique, because post-composing $\hat{f}_0$ by any element of the group $\langle\tau_0\rangle$ provides another homeomorphism with the same properties. The induced action of this operation on $\hat{f}_{0,\zeta}$ consists of multiplying $\hat{f}_{0,\zeta}$ by a $d$-th root of unity.
\end{itemize}
\end{Remark}

\subsection{Projection to moduli space and topology of $\strate$}\label{sec:proj:mod:sp}
Given $(X,q) \in \strate$, let $x_1,\dots,x_n$ denote the zeros of $q$ such that the order of $x_i$ is $k_i$ (a zero of negative order is a pole). Let $Z(q)$ denote the set $\{x_1,\dots,x_n\}$.
We have a natural forgetful map  from $F: \strate \ra \Mcal_{g,n}$ which associates to $(X,q)$ the Riemann surface with marked points $(X,\{x_1,\dots,x_n\})$. Let us denote by $\pstrate$ the image of $\strate$ in $\Mcal_{g,n}$ under this map.
A point $(X,\{x_1,\dots,x_n\})\in \pstrate$ is characterized by the following property: the divisor $\sum_{i=1}^n k_ix_i$ on $X$ is equivalent to $K_X^{\otimes d}$.
In particular, $\pstrate$ is a subvariety of $\Mcal_{g,n}$.
Since there is at most one meromorphic $d$-differential on $X$, up to  multiplication by a scalar, whose divisor is $\sum_{i=1}^n k_ix_i$, we see that $\pstrate$  can be identified with $\Pb\strate$.

\medskip

Let $\Ccal_{g,n}$ be the universal curve over $\Mcal_{g,n}$. Let $K_{\Ccal_{g,n}/\Mcal_{g,n}}$ denote the relative canonical line bundle of the projection  $p: \Ccal_{g,n} \ra \Mcal_{g,n}$. There exist $n$ sections $\sigma_i: \Mcal_{g,n} \ra \Ccal_{g,n}, \; i=1,\dots,n$, of $p$ such that if $m\in \Mcal_{g,n}$ represents the pointed curve $(X,\{x_1,\dots,x_n\})$, then $\sigma_i(m)$ is the point in $p^{-1}(m)\subset \Ccal$ which corresponds to $x_i$ under the identification  $p^{-1}(m) \simeq X$.
Let $D_i$ denote the image of $\Mcal_{g,n}$ under $\sigma_i$, then $D_i$ is a divisor of $\Ccal_{g,n}$. Let $\Lcal_i$ denote the line bundle associated with $D_i$. Consider the line bundle
$$
{\mathcal K}:=K^{\otimes d}_{\Ccal_{g,n}/\Mcal_{g,n}}\otimes\Lcal^*_1{}^{\otimes k_1}\otimes\dots\otimes\Lcal_n^*{}^{\otimes k_n}
$$
on $\Ccal_{g,n}$. By definition, if $m=(X,\{x_1,\dots,x_n\}) \in \pstrate$, then the restriction of ${\mathcal K}$ to the fiber over $m$ is trivial. Thus $p_*(\mathcal{K}_{|p^{-1}(\pstrate)})$ is a line bundle over $\pstrate$, whose fiber over $m$ is generated by any $d$-differential $q$ on $X$ such that $\mathrm{div}(q)=\sum_{i=1}^nk_ix_i$.
%By definition, a holomorphic section of $p_*(\mathcal{K})$ over an open subset $\Ucal$ of $\pstrate$ corresponds to a holomorphic section of $\mathcal{K}$ over $p^{-1}(\Ucal)$.
Note that $\strate$ is the complement of the zero section in the total space of this line bundle.
Thus this description provides us with a natural topology and a complex structure on $\strate$.

\subsection{Lifting to Abelian differentials}\label{sec:lift:ab:diff}
Let $\hat{g}$ be the genus of $\hat{M}$, and $\hat{n}=\card(\hat{\Sig})$. Denote by $\{\hat{s}_1,\dots,\hat{s}_{\hat{n}}\}$ the points in $\hat{\Sig}$.
Let $\lstrate$ denote the moduli space of triples $(\hat{X},\hat{\omega},\tau)$, where $\hat{X}$ is a Riemann surface of genus $\hat{g}$, $\hat{\omega}$ is a holomorphic $1$-form on $\hat{X}$, and $\tau: \hat{X} \ra \hat{X}$ is an automorphism of order $d$ of $\hat{X}$ such that
\begin{itemize}
\item[(i)] $\tau^*\hat{\omega}=\zeta\hat{\omega}$,

\item[(ii)] $X:=\hat{X}/\langle \tau \rangle$ is a Riemann surface of genus $g$.

\item[(iii)] there is a meromorphic $d$-differential $q$ on $X$ such that $(X,q) \in \strate$ and $\varpi^*q=\hat{\omega}^d$, where $\varpi: \hat{X} \ra X$ is a the canonical projection.
\end{itemize}
Let $\hat{f}: \hat{M} \to \hat{X}$ be a homeomorphism as in Lemma~\ref{lm:ab:diff:eig:vect}.
Define $\hat{Z}(\hat{\omega})=\hat{f}(\hat{\Sig})$.
By construction, the  automorphism $\tau$ acts freely on $\hat{X}\setminus\hat{Z}(\hat{\omega})$, and $\varpi(\hat{Z}(\hat{\omega}))$ is the set of zeros and poles of $q$.

Let $\hat{x}_j=\hat{f}(\hat{s}_j), \; j=1,\dots,\hat{n}$.
The order  $\hat{k}_j$ of $\hat{\omega}_0$ at $\hat{x}_j$ can be computed from  the order $k_i$ of $q_0$ at $\varpi(\hat{x}_j)$ as follows
\begin{equation}\label{eq:order:1:form:cov}
\hat{k}_j=d_i(1+\frac{k_i}{d})-1
\end{equation}
where $d_i$ is the order of $k_i$ in $\Z/d\Z$. Note that $\hat{k}_j=0$ if and only if there exist positive integers $n_i,d_i$ such that $d=n_id_i$ and $k_i=n_i(1-d_i)$.
By a slight abuse, we will call $\hat{Z}(\hat{\omega})$ the zero set of $\hat{\omega}$.

\medskip

If  $\zeta'$ is any $d$-th root of unity, then the triples $(\hat{X},\zeta'\hat{\omega},\tau)$ and $(\hat{X},\hat{\omega},\tau)$ represent the same element of $\lstrate$.
That is because there is $k\in \{0,\dots,d-1\}$ such that $\zeta'=\zeta^k$, and $\tau^k: \hat{X} \to \hat{X}$ is an isomorphism which satisfies $\tau^{k*}\hat{\omega}=\zeta'\hat{\omega}$, and $\tau^{-k}\circ\tau\circ\tau^k=\tau$.
Therefore, as a direct consequence of Lemma~\ref{lm:ab:diff:eig:vect}, we get

\begin{Corollary}\label{cor:P:bijection}%\label{cor:P:surjective}
The map
$$
\begin{array}{cccc}
 \PP: & \lstrate & \ra & \strate \\
      & (\hat{X},\hat{\omega},\tau) & \mapsto & (X,q)
\end{array}
$$
is a bijection.
\end{Corollary}
By Corollary~\ref{cor:P:bijection}, we can endow $\lstrate$ with the topology of $\strate$.

Set $\hat{\kappa}=(\hat{k}_1,\dots,\hat{k}_{\hat{n}})$.
Let $\lstrateabdiff$ denote the stratum  of holomorphic $1$-forms $(\hat{X},\hat{\omega})$, where $\hat{X}$ is a Riemann surface of genus $\hat{g}$, and $\hat{\omega}$ has $\hat{n}$ zeros with orders given by $\hat{\kappa}$.  %Note that the zeros of $\hat{\omega}$ are not numbered.
There is a forgetful map $\ff: \lstrate \ra \lstrateabdiff, (\hat{X},\hat{\omega},\tau) \mapsto (\hat{X},\hat{\omega})$ which is finite to one (given a pair $(\hat{X},\hat{\omega})$ there may be more than one automorphism $\tau$ such that $(\hat{X},\hat{\omega},\tau) \in \lstrate$).
It is a well known fact that there exist some finite branched coverings $\Omega\widetilde{\Mcal}^{\langle \zeta \rangle}_{g,n}(\kappa)$ and $\Omega\widetilde{\Mcal}_{\hat{g}}(\hat{\kappa})$ of $\lstrate$ and of $\lstrateabdiff$ respectively such that $\ff$ lifts to an embedding $\tilde{\ff}: \Omega\widetilde{\Mcal}^{\langle \zeta \rangle}_{g,n}(\kappa) \ra \Omega\widetilde{\Mcal}_{\hat{g},\hat{n}}(\hat{\kappa})$.
This means that locally, up to taking some finite order coverings, we can identify a  neighborhood of an element $(\hat{X},\hat{\omega},\tau)$ in $\lstrate$ with a subset of a neighborhood of $(\hat{X},\hat{\omega})$ in $\lstrateabdiff$. As a consequence, we get

\begin{Proposition}\label{prop:d-th:root:hol}
Let $(X_0,q_0)$ and $(\hat{X}_0,\hat{\omega}_0,\tau_0)$ be as in Lemma~\ref{lm:ab:diff:eig:vect}.
Then there is a neighborhood $\Vcal$ of $(X_0,q_0)$ in $\strate$ and a map
 $$
 \begin{array}{cccc}
  \RR: & \Vcal & \ra & \lstrateabdiff\\
       & (X,q) & \mapsto & (\hat{X},\hat{\omega})
 \end{array}
 $$
which is biholomorphic onto its image such that $\RR((X_0,q_0))= (\hat{X}_0,\hat{\omega}_0)$, and if $(\hX,\homg)=\RR((X,q))$, then there is an automorphism $\tau$ of $\hX$ such that $(\hX,\homg,\tau)$ is the canonical cyclic cover of $(X,q)$.
\end{Proposition}

\section{Local coordinates by period mappings}\label{sec:loc:coord:per}
In this section,  we introduce the period mappings on $\lstrate$ and show that they form an atlas which defines a structure of affine orbifold on $\strate$. The main results of this section (Proposition~\ref{prop:eigen:sp:loc:coord}, Corollary~\ref{cor:dim:Vzeta}, Corollary~\ref{cor:aff:mfd:structure}) have been shown in \cite{Veech:FS} and \cite{BCGGM2} (see also \cite{Sch16}).
The proofs we present here are different from the ones in the work mentioned above. In particular, we will make use of triangulations of $\hat{M}$ that are invariant under $T_\zeta$.

\subsection{Admissible triangulations}\label{sec:loc:chart:tria}
Let $\Tcal$ be a topological triangulation of $M$ with vertex set $\Sig$.
Let $\hat{\Tcal}$ denote the triangulation of $\hat{M}$ that is the pullback of $\Tcal$.
We fix an orientation for every edge of $\hat{\Tcal}$.

Let $N_1$ and $N_2$ be the numbers of edges and triangles of $\hat{\Tcal}$ respectively. We identify a vector $v\in \C^{N_1}$ with a function $v: \hat{\Tcal}^{(1)} \ra \C$.
We have an action of $T_\zeta$ by permutations on the sets $\hat{\Tcal}^{(1)}$ and $\hat{\Tcal}^{(2)}$ (since $\Tcal=\hat{\Tcal}/\langle T_\zeta \rangle$).
Note that the action of $T_\zeta$ on $\hat{\Tcal}^{(1)}$ is free if $d>1$.
Consider the system $(\Scal)$ of $N_1+N_2$ linear equations which are defined as follows
\begin{itemize}
\item each triangle $\theta \in \hat{\Tcal}^{(2)}$, whose sides are denoted by $e_1,e_2,e_3$, corresponds to an equation of the form
\begin{equation}\label{eq:lin:equa:tria}
  \pm v(e_1) \pm v(e_2) \pm v(e_3)=0
\end{equation}
where the signs $\pm$ are determined according to the orientation of $e_1,e_2,e_3$.

\item each edge $e\in \hat{\Tcal}^{(1)}$ corresponds to an equation of the form
\begin{equation}\label{eq:lin:equa:sym}
  v(T_\zeta(e))-\zeta v(e)=0.
\end{equation}
\end{itemize}
Let $V\subset \C^{N_1}$ denote the space of solutions of $(\Scal)$.
We will also consider the system $(\Scal_1)$ (resp. $(\Scal_2)$)  of all linear equations of type  \eqref{eq:lin:equa:tria} (resp. of type \eqref{eq:lin:equa:sym}). Let $V_1$ and $V_2$ denote the space of solutions of $(\Scal_1)$ and $(\Scal_2)$ respectively. By definition, $V=V_1\cap V_2$.

\begin{Lemma}\label{lm:sys:S:eigen:sp}
Let $V_\zeta=\ker(T_\zeta-\zeta\Id) \subset H^1(\hat{M},\hat{\Sig},\C)$.
Then $V_\zeta$ is isomorphic to the subspace $V \subset \C^{N_1}$.
\end{Lemma}
\begin{proof}
 We first notice that $\hat{\Tcal}$ provides us with a cell complex structure on $\hat{M}$.
 Thus, the system $(\Scal_1)$ defines the space $H^1_{\hat{\Tcal}}(\hat{M},\hat{\Sig},\C) \simeq H^1(\hat{M},\hat{\Sig},\C)$.
Since the solutions of the system $(\Scal_2)$ correspond precisely to the vectors in $\C^{N_1}$ such that $T_\zeta(v)=\zeta v$, the lemma follows.
\end{proof}

\subsection{Period mappings}\label{sec:per:map:loc:chart}
Let $(X_0,q_0)$ be an element of $\strate$. Fix a homeomorphism $f_0: (M,\Sig) \to (X_0,Z(q_0))$.
Let $(\hat{X}_0,\hat{\omega}_0,\tau_0)$ and $\hat{f}_0: (\hat{M},\hat{\Sigma})\to (\hat{X}_0,\hat{Z}(\hat{\omega}_0))$ be as in  Lemma~\ref{lm:ab:diff:eig:vect}.
Given a $\C$-valued closed $1$-form  $\xi$ on $\hat{M}$, we will denote by $[\xi]$ its cohomology class in $H^1(\hat{M},\hat{\Sig},\C)$.

Assume that $(\hat{X}_0,\hat{\omega}_0)$ is not an orbifold point of $\lstrateabdiff$. There exists a neighborhood $\Wcal$ of $(\hat{X}_0,\hat{\omega}_0)$ in $\lstrateabdiff$ such that for any $(\hat{X},\hat{\omega}) \in \Wcal$, we have a canonical homeomorphism $\hat{f}: (\hat{M},\hat{\Sigma}) \to (\hat{X},\hat{Z}(\hat{\omega}))$ defined up to isotopy such that if $(\hat{X},\hat{\omega})=(\hat{X}_0,\hat{\omega}_0)$ then $\hat{f}=\hat{f}_0$.
Note that there always exists a diffeomorphism in the isotopy class of $\hat{f}$ (see \cite{FM12}), therefore we can assume that $\hat{f}$  is itself a diffeomorphism.
Consequently, we have a well defined map
$$
\begin{array}{cccc}
\Phi: & \Wcal & \ra & H^1(\hat{M},\hat{\Sig},\C),\\
     & (\hat{X},\hat{\omega}) & \mapsto & [\hat{f}^*\hat{\omega}]
\end{array}
$$
The map $\Phi$ is called a {\em period mapping}.
It is a well known fact that $\Phi$ is a (holomorphic) local chart of $\lstrateabdiff$ if $\Wcal$ is small enough (see for instance \cite{MT02, Z:survey}).
%From now on, we will suppose that $\Phi$ is a local chart of $\lstrateabdiff$ on $\Wcal$.

\begin{Proposition} \label{prop:eigen:sp:loc:coord}
Assume that $(X_0,q_0)$ is not an orbifold point of $\strate$, and $(\hat{X}_0,\hat{\omega}_0)$ is not an orbifold point of $\lstrateabdiff$.
Let $\RR: \Vcal \ra \lstrateabdiff$ be the map in Proposition~\ref{prop:d-th:root:hol}.
Then for $\Vcal$ small enough,  $\Xi:=\Phi\circ\RR: \Vcal \ra V_\zeta$ realizes a biholomorphic map from $\Vcal$ onto  an open subset of $V_\zeta$.
As a consequence, $\Xi$ is a holomorphic local chart for $\strate$.
\end{Proposition}

\begin{proof}
Let $\Tb_0$ be a triangulation  of $(X_0,q_0)$ whose vertex set is $Z_0=Z(q_0)$, and all the edges are geodesics of the flat metric defined by $q_0$.
It is a classical result that such triangulations always exist (see for instance \cite{MS91}).
In what follows $\Tcal$ will be the triangulation of $M$ induced from $\Tb_0$ via $f_0$.

We identify $H^1(\hat{M},\hat{\Sig},\C)$ (resp.  $V_\zeta$) with $V_1$ (resp. $V$) via the map $\eta \mapsto \{\eta(e), \; e \in \hat{\Tcal}^{(1)}\}$.
Define the vector $v_0 \in \C^{N_1}$ by
$$
v_0(e)=\int_{\hat{f}_0(e)}\hat{\omega}_0.
$$
Recall that locally $\lstrate$ can be considered as subset of $\lstrateabdiff$.
Since we have $\Phi\circ\RR((X_0,q_0))=v_0\in V$, it is enough to show that $\Phi(\Wcal\cap\lstrate)$ is a neighborhood of $v_0$ in $V$.

For any triangle $\theta \in \hat{\Tcal}^{(2)}$ and any vector $v \in V$, let $\theta(v)$ denote the triangle in the plane which is formed by the vectors $v(e_1), v(e_2), v(e_3)$, where $e_1,e_2,e_3$ are the sides of $\theta$. We now consider an open neighborhood $U$ of $v_0$ in $V$ such that,
\begin{itemize}
 \item[(a)] for every $v\in U$ and every triangle $\theta \in \hat{\Tcal}^{(2)}$, there is an orientation preserving affine map of the plane that sends $\theta(v_0)$ to $\theta(v)$,

 \item[(b)] for all $k \in \{1,\dots,d-1\}$, $U\cap\zeta^k\cdot U=\vide$.
\end{itemize}
For any fixed $v\in U$, we can glue the triangles $\{\theta(v), \; \theta \in \hat{\Tcal}^{(2)}\}$ to get a flat surface with conical singularity. Note that the holonomies of this flat surface  are all translations, therefore the surface obtained is a translation surface which is defined by  a holomorphic $1$-form $(\hat{X}_v,\hat{\omega}_v)$.

For any $\theta\in \hat{\Tcal}^{(2)}$, there is a unique affine map $A_{\theta}(v_0,v): \R^2 \ra \R^2$  which maps the triangle $\theta(v_0)$ onto the triangle $\theta(v)$ sending each side of $\theta(v_0)$ to the side of $\theta(v)$ corresponding to the same edge of $\hat{\Tcal}$. The family of maps $\{A_{\theta}(v_0,v),\; \theta \in \Tcal^{(2)}\}$ defines a homeomorphism $h_v: \hat{X}_0 \ra \hat{X}_v$.

Let $\hat{f}_v:=h_v\circ\hat{f}_0: \hat{M} \ra \hat{X}_v$. By construction, every edge  $e \in \hat{\Tcal}^{(1)}$ is mapped to a geodesic arc of the flat metric defined by $\hat{\omega}_v$, and $v(e)=\int_{\hat{f}_v(e)}\hat{\omega}_v$.
Let $\tau_v=\hat{f}_v\circ T_\zeta\circ \hat{f}_v^{-1}: \hat{X}_v \ra \hat{X}_v$.

For any $\theta \in \hat{\Tcal}^{(2)}$, let $\theta'\in \hat{\Tcal}^{(2)}$ be the image of $\theta$ by $T_\zeta$.
Since we have $v(T_\zeta(e))=\zeta v(e)$ for all $e\in \hat{\Tcal}^{(1)}$ (because $v$ satisfies $(\Scal_2)$), it follows that   $\theta'(v)=\zeta\cdot \theta(v)$, where $\zeta\cdot$ means the rotation of $\R^2\simeq \C$ corresponding to the multiplication by $\zeta$. Consequently, the homeomorphism $\tau_v: \hat{X}_v \ra \hat{X}_v$ is actually an isometry of the flat metric defined by $\hat{\omega}_v$ on $\hat{X}_v$, it satisfies in particular $\tau^*_v\hat{\omega}_v=\zeta\hat{\omega}_v$. This implies that $(\hat{X}_v,\hat{\omega}_v,\tau_v)$ is an element of $\lstrate$.
Clearly, we have $\Phi((\hat{X}_v,\hat{\omega}_v,\tau_v)) =v$. Thus $\Phi(\Wcal\cap\lstrate)$ contains a neighborhood of $v_0$ in $V$ and the proposition follows.
\end{proof}

\begin{Remark}\label{rmk:orb:pts}
In general, $\strate$ and $\lstrateabdiff$ are not manifolds, so Proposition~\ref{prop:eigen:sp:loc:coord} does not apply to every point of $\strate$. Nevertheless, it is a well known fact that $\strate$ and $\lstrateabdiff$ admit finite coverings that are manifolds. Therefore, there exists some finite covering $\widetilde{\strate}$ of $\strate$ of which  Proposition~\ref{prop:eigen:sp:loc:coord} applies to {\em every} point, that is the map $\Xi$ defines  a holomorphic local chart for $\widetilde{\strate}$ in a neighborhood of its every point.
%In particular,  $\widetilde{\strate}$ is a manifold.

In what follows, we will implicitly be working with $\widetilde{\strate}$.
However, to lighten the discussion,  by a slight abuse we will use $\Xi$ as local charts in a neighborhood of every point in $\strate$.
\end{Remark}

An immediate consequence of Proposition~\ref{prop:eigen:sp:loc:coord} is that $\dim V_\zeta=\dim V= \dim \strate$. Thus by the works \cite{Veech:FS, BCGGM2, Sch16}

\begin{Corollary}\label{cor:dim:Vzeta}%\label{prop:dim:Vzeta}
We have
$$
\dim_\C V=\left\{
\begin{array}{ll}
2g+n-1 & \text{ if } d=1, \\
2g+n-2 & \text{ if } d \geq 2.
\end{array}
\right.
$$
\end{Corollary}
In the Appendix \ref{sec:proof:dim:Vzeta}, we give an independent proof of this fact using exclusively the triangulation $\hat{\Tcal}$ of $\hat{M}$ and the action of $T_\zeta$.

\subsection{Switching the marking}\label{sec:marking}
To define $\Xi$, one needs to specify a homeomorphism $f_0: (M,\Sig) \ra (X_0,Z(q_0))$.
The homotopy class of $f_0$ will be referred to as a {\em marking}.
The maps $\Xi$ in Proposition~\ref{prop:eigen:sp:loc:coord} provide us with an atlas for $\strate$.
Transition maps of the atlas correspond to changes of markings.

Consider now another homeomorphism $f'_0: M \ra X_0$ such that $f'_0(s_i) =x^0_i, \; i=1,\dots,n$. Let $\Gamma':=\ker(\veps_0\circ f'_{0*}) \subset \pi_1(M',*)$, and $\pi': \hat{N}'\ra M'$ be the covering associated with $\Gamma'$. There is a compact surface $\hat{N}$ and a finite subset $\hat{\Pi}\subset \hat{N}$ such that $\hat{N}'\simeq \hat{N}\setminus\hat{\Pi}$ and $\pi'$ extends to a branched covering $\pi':\hat{N} \ra M$.
By definition, $h:={f'_0}^{-1}\circ f_0: M \ra M$ is a homeomorphism which is identity on the subset $\Sig$. In particular, $h$ restricts to a homeomorphism of $M'$.

Since we have $h_*(\Gamma)=\Gamma'$, there exists a homeomorphism $\hat{h}: \hat{M}' \ra \hat{N}'$ which lifts  $h$, that is we have the following commutative diagram
\begin{center}
\begin{tikzpicture}[scale=0.4]
\node (A) at (0,4)  {$\hat{M}'$};

\node (B) at (6,4)  {$\hat{N}'$};

\node (C) at (0,0) {$M'$};

\node (D) at (6,0) {$M'$};

\path (A) edge[->, >=angle 60, font=\scriptsize] node[above] {$\hat{h}$} (B)
      (A) edge[->, >=angle 60, font=\scriptsize] node[left] {$\pi$} (C)
      (B) edge[->, >=angle 60, font=\scriptsize] node[right] {$\pi'$} (D)
      (C) edge[->, >=angle 60, font=\scriptsize] node[below] {$h$} (D);

\end{tikzpicture}
\end{center}
The map $\hat{h}$ extends to a homeomorphism $\hat{h}: \hat{M} \ra \hat{N}$ sending $\hat{\Sig}$ onto $\hat{\Pi}$, and hence induces an isomorphism $\hat{h}^*: H^1(\hat{N},\hat{\Pi},\C) \ra H^1(\hat{M},\hat{\Sig},\C)$. Note that $\hat{h}^*$ restricts to  isomorphisms between $H^1(\hat{N},\hat{\Pi},\Z)$ and $H^1(\hat{M},\hat{\Sig},\Z)$, and between $H^1(\hat{N},\Z)$ and $H^1(\hat{M},\Z)$ respectively.

Let $\alpha'\in \pi_1(M',*)$ be an element such that $\veps_0(f'_{0*}(\alpha'))=\zeta \in  \Z/d\Z$. Then $\alpha'$ gives rise to a homeomorphism $T'_\zeta: \hat{N} \ra \hat{N}$ preserving the set $\hat{\Pi}$ which satisfies $\pi'\circ T'_\zeta=\pi'$, and ${T'_\zeta}^d=\id$. Let $V'_\zeta:=\ker(T'_\zeta -\zeta\Id) \subset H^1(\hat{N},\hat{\Pi},\C)$.

Since $T_\zeta$ (resp. $T'_\zeta$) preserves $H^1(\hat{M},\hat{\Sig},\Z)$ (resp. $H^1(\hat{N},\hat{\Pi},\Z)$), there is a basis of $V_\zeta$ (resp. of $V'_\zeta$) consisting of elements in $V_\zeta\cap H^1(\hat{M},\hat{\Sig},\Q(\zeta))$ (resp. in $V'_\zeta\cap H^1(\hat{N},\hat{\Pi},\Q(\zeta))$).
\begin{Lemma}\label{lm:marking:change}
The map $\hat{h}^*$ restricts to an isomorphism $\hat{h}^*: V'_\zeta \ra V_\zeta$ which maps $V'_\zeta\cap H^1(\hat{N},\hat{\Pi},\Q(\zeta))$ onto $V_\zeta\cap H^1(\hat{M},\hat{\Sig},\Q(\zeta))$.
\end{Lemma}
\begin{proof}
 Let $\alpha =h_*^{-1} (\alpha') \in \pi_1(M',*)$. By assumption, we have $\veps_0(f_{0*}(\alpha))=\veps_0(f'_{0*}(\alpha'))=\zeta$. Thus $\alpha$ gives rise to the homeomorphism $T_\zeta: \hat{M}\ra \hat{M}$.
 By construction, we have $\hat{h}\circ T_\zeta= T'_\zeta\circ \hat{h}$. Thus $\hat{h}^*$ restricts to an isomorphism between $\ker(T'_\zeta-\zeta\Id)\subset H^1(\hat{N},\hat{\Pi},\C)$ and $\ker(T_\zeta-\zeta\Id)\subset H^1(\hat{M},\hat{\Sig},\C)$.
The last assertion follows from the fact that $\hat{h}^*$ maps $H^1(\hat{N},\hat{\Pi},\Q(\zeta))$ onto $H^1(\hat{M},\hat{\Sig},\Q(\zeta))$.
\end{proof}

\medskip

In the case $\Gamma'=\Gamma$, the surfaces $\hat{N}$ and $\hat{M}$ are identified, and we have
\begin{Lemma}\label{lm:spec:marking:change}
Suppose that $\Gamma'=\Gamma$. Then there exists a generator $\zeta'$ of $\Z/d\Z$ such that $V'_\zeta = V_{\zeta'} \subset H^1(\hat{M},\hat{\Sig},\C)$, where $V_{\zeta'}=\ker(T_\zeta-\zeta'\Id)$. Furthermore, $V'_\zeta=V_\zeta$ if and only if $\veps_0\circ f'_{0*}=\veps_0\circ f_{0*}$.
\end{Lemma}
\begin{proof}
By definition, we have $\pi_1(M',*)/\Gamma\simeq \Z/d\Z$ with the identification given by $\veps_0\circ f_{0*}$.
Let $\alpha\in \pi_1(M',*)$ be an element such that $\veps_0(f_{0*}(\alpha))=\zeta$ and $\alpha'=h_*(\alpha)$. Note that $\veps_0(f'_{0*}(\alpha'))=\veps_0(f_{0*}(\alpha))=\zeta$.
By assumption, $\alpha$ is a generator of $\pi_1(M',*)/\Gamma\simeq \Z/d\Z$, there exists $k\in \Z$ such that $\alpha'=\alpha^k \mod \Gamma$.
Since $h_*: \pi_1(M',*) \ra \pi_1(M',*)$ is an isomorphism which preserves $\Gamma$, it induces an isomorphism of $\Z/d\Z$.
Therefore, $\alpha'$ is also a generator of $\pi_1(M',*)/\Gamma$, which means that $\gcd(k,d)=1$.
By definition, $T_\zeta$ and $T'_\zeta$ are the automorphisms  of the covering $\pi: \hat{M}'\ra M'$ associated with $\alpha$ and $\alpha'$ respectively. Therefore, we have $T'_\zeta= T^k_{\zeta}$.

Let $\zeta'=h^{-1}_*(\zeta)$. Then $\zeta'$ is a generator of $\Z/d\Z$ and ${\zeta'}^k=\zeta$.
Now
$$
V'_\zeta:=\ker(T'_{\zeta}-\zeta\Id)=\ker(T^k_\zeta-\zeta\Id)=\ker(T_\zeta - \zeta'\Id)=V_{\zeta'}
$$
and the first assertion follows. For the second assertion, it is enough to observe that $\zeta'=\zeta$ if and only if $h_*$ is identity on $\Z/d\Z$.
\end{proof}

An immediate consequence of Lemma~\ref{lm:marking:change} is the following

\begin{Corollary}~\label{cor:aff:mfd:structure}
 The atlas given by the maps $\Xi$ in Proposition~\ref{prop:eigen:sp:loc:coord} defines a structure of affine complex orbifold on $\strate$.
\end{Corollary}

%*****************************************************
%*****************************************************
%*****************************************************
%*****************************************************

\section{Projection to absolute cohomology}\label{sec:proj:abs:coh}
Let $(X_0,q_0) \in \strate$.  Let $Z_0=\{x^0_1,\dots,x^0_n\}$ denote the set of zeros and poles of $q_0$, the order of $q_0$ at $x^0_i$ is $k_i$. Let $f_0: M \ra X_0$ be a homeomorphism sending $s_i$ to $x^0_i, \; i=1,\dots,n$ (the homotopy class of the map $f_0$ is a marking of $(X_0,Z_0)$). Let $\veps_0, \Gamma, \hat{M},\hat{\Sig}, T_\zeta$ be as in Section~\ref{sec:top:prelim}.

Let $\pp: H^1(\hat{M},\hat{\Sig},\C) \ra H^1(\hat{M},\C)$ be the natural projection, that is for any $\eta \in  H^1(\hat{M},\hat{\Sig},\C)$,  $\pp(\eta)$ is the restriction of $\eta$ to the cycles in $H_1(\hat{M},\Z)$.
We have the following exact sequence
\begin{equation}\label{eq:coh:ex:seq:gen}
0 \to H^0(\hM,\C) \to H^0(\hSig,\C) \to H^1(\hM,\hSig,\C) \overset{\pp}{\to} H^1(\hM,\C) \to 0.
\end{equation}
Since $T_\zeta$ is a homeomorphism of $\hM$ preserving the set $\hSig$, its actions on the cohomology spaces in \eqref{eq:coh:ex:seq:gen} are equivariant.
%It is not difficult to check that the maps in \eqref{eq:coh:ex:seq:gen} are equivariant with respect to the actions of $T_\zeta$ on the cohomology spaces in this sequence.
Thus, using the fact that $T_\zeta$ has finite order,  we get the following exact sequence
\begin{equation}\label{eq:coh:ex:seq:eig:sp}
0 \to  H^0(\hM,\C)_\zeta \to H^0(\hSig,\C)_\zeta \to H^1(\hM,\hSig,\C)_\zeta=V_\zeta \overset{\pp}{\to} H^1(\hM,\C)_\zeta \to 0
\end{equation}
where the subscript $\bullet_\zeta$ means the $\zeta$-eigenspace of the action of $T_\zeta$ on the corresponding space.
As an immediate consequence, we get
\begin{Lemma}\label{lm:proj:coh:surj}
Let $H_\zeta=\ker(T_\zeta -\zeta\Id) \subset H^1(\hat{M},\C)$.
Then we have $H_\zeta=\pp(V_\zeta)$. %where $\pp: H^1(\hat{M},\hat{\Sig},\C) \ra H^1(\hat{M},\C)$.
\end{Lemma}

Let $r=\card\left(\{k_1,\dots,k_n\}\cap(d\Z)\right)$.
We can suppose that $k_i\in d\Z$ if and only if $i \in \{1,\dots,r\}$.
Our goal now is to show
\begin{Proposition}\label{prop:ker:project:coh}
For $i=1,\dots,r$, pick a point $\hat{s}_{i}$ in $\pi^{-1}(\{s_i\})$.
Let $c_i$ be a path from $\hat{s}_{i}$ to $T_\zeta(\hat{s}_{i})$.
If $d\geq 2$, then we have
$$
\dim \ker\pp\cap V_\zeta=r,
$$
and there is a basis $\{\eta_1,\dots,\eta_r\}$ of $\ker\pp\cap V_\zeta$ such that $\eta_j(c_i)=(\zeta-1)\delta_{ij}$.
\end{Proposition}
\begin{Remark}\label{rmk:Veech:ker:proj}\hfill
\begin{itemize}
 \item This result has been known to Veech (see~\cite[Sect. 8]{Veech:FS}), we will provide here an independent proof adapted to our setting.

 \item In the case $d=1$ we have $r=n$, $(\hat{M},\hat{\Sig})=(M,\Sig)$, and $V_\zeta=H^1(M,\Sig,\C)$. In particular, $\dim\ker\pp=n-1=r-1$.
\end{itemize}
\end{Remark}
\begin{proof}
%From now on we suppose that $d>1$, which means that $\zeta \neq 1$.
Since $d>1$, we have $\zeta\neq 1$, hence $H^0(\hM,\C)_\zeta=\{0\}$.
Thus the exact sequence \eqref{eq:coh:ex:seq:eig:sp} implies the following
\begin{equation}\label{eq:coh:ex:seq:eig:sp:red}
 0 \to  H^0(\hSig,\C)_\zeta \to V_\zeta \overset{\pp_{|V_\zeta}}{\to} H_\zeta \to 0
\end{equation}
which means that $\ker\pp\cap V_\zeta=\ker \pp_{|V_\zeta} = H^0(\hSig,\C)_\zeta$.

For $i=1,\dots,n$, let $\del_i$ be an element of $\pi_1(M',*)$ represented by a loop freely homotopic to the boundary of a small disc about $s_i$.
Since $x^0_i$ is a zero of order $k_i$ of $q_0$, we have $\veps_0\circ f_{0*}(\del_i)=e^{\frac{2\pi\imath}{d}k_i}\simeq k_i \in \Z/d\Z$.
Let $n_i=\gcd(d,k_i)$. Note that the order of $k_i$ in $\Z/d\Z$ is $\frac{d}{n_i}$.
It follows that $n_i=\card(\pi^{-1}(\{s_i\}))$, and $T_\zeta^{n_i}$ is identity on the set $\pi^{-1}(\{s_i\})$.
Choose a point $\hat{s}_i$ in $\pi^{-1}(\{s_i\})$, then $\pi^{-1}(\{s_i\})=\{T^j_\zeta(\hat{s}_i), \; j=0,\dots,n_i-1\}$.
% and $d_i=\frac{d}{n_i}$. Note that $d_i$ is the order of $k_i$ in $\Z/d\Z$.
% In particular, we have $n_i=\card(\pi^{-1}(\{s_i\}))$.

Consider an element $\eta \in H^0(\hSig,\C)_\zeta$. We first show that $\eta(\hat{s}_i)=0$ for all $i=r+1,\dots,n$. Indeed, by assumption, we have $k_i \not\equiv 0 \mod d$. Thus $n_i < d$. Hence
$$
\zeta^{n_i}\eta(\hat{s}_i)=T_\zeta^{n_i*}\eta(\hat{s}_i)=\eta(T^{n_i}_\zeta(\hat{s}_i))=\eta(\hat{s}_i).
$$
Since $\zeta^{n_i}\neq 1$, we must have $\eta(\hat{s}_i)=0$.

Since $\eta(T_\zeta(\hat{s}))=\zeta\eta(\hat{s})$ for all $\hat{s}\in \hSig$,  we see that $\eta$ is uniquely determined by $(\eta(\hat{s}_1),\dots,\eta(\hat{s}_r))$.
Therefore, a basis of $H^0(\hSig,\C)_\zeta$ is given by $\{\eta_1,\dots,\eta_r\}$, where for $i\in\{1,\dots,n\}, \; k\in \{0,\dots,d-1\}$,
$$
\eta_j(T^k(\hat{s}_i))= \left\{
\begin{array}{cl}
\zeta^k & \hbox{ if $i=j$}, \\
0 & \hbox{ otherwise}.
\end{array}
\right.
$$
Consider $\{\eta_1,\dots,\eta_r\}$ as elements of $H^1(\hM,\hSig,\C)$. We have
$$
\eta_j(c_i)=\eta_j(T_\zeta(\hat{s}_i))-\eta_j(\hat{s}_i)=(\zeta-1)\delta_{ij}
$$
as desired.
\end{proof}

\section{Volume form}\label{sec:vol:form}
\subsection{The intersection form}\label{sec:intersect:form}
On $H^1(\hat{M},\C)$ we have a natural Hermitian form $(.,.)$ defined as follows: let $(a_1,\dots,a_{\hat{g}},b_1,\dots,b_{\hat{g}})$ be a symplectic basis of $H_1(\hat{M},\Z)$. For $\eta,\mu \in H^1(\hat{M},\C)$, we have
\begin{equation}\label{eq:inter:form:def}
(\eta,\mu)=\frac{\imath}{2}\sum_{j=1}^{\hat{g}} \left(\eta(a_j)\overline{\mu(b_j)} - \eta(b_j)\overline{\mu(a_j)}\right).
\end{equation}
It is well known that $(.,.)$ has signature $(\hat{g},\hat{g})$, and is preserved by all homeomorphisms of $\hat{M}$.

\begin{Lemma}\label{lm:inter:form:non:degen}
The restriction of $(.,.)$ to $H_\zeta$ is non-degenerate.
\end{Lemma}
\begin{proof}
Let $H_\zeta^\bot$ denote the orthogonal complement of $H_\zeta$ with respect to $(.,.)$.  Since $(.,.)$ is non degenerate on $H^1(\hat{M},\C)$, we have $\dim H^1(\hat{M},\C)=\dim H_\zeta+ \dim H_\zeta^\bot$.
Let $\lambda_0=\zeta, \lambda_1,\dots,\lambda_\ell$ be the eigenvalues of $T_\zeta$.
Since $T_\zeta$ has finite order, it is diagonalizable, hence we can write
$$
H^1(\hat{M},\C)=\oplus_{i=0}^\ell E_{\lambda_i}
$$
where $E_{\lambda_i}$ is the eigenspace of $\lambda_i$.
Since $T_\zeta$ preserves $(.,.)$, we have $E_{\lambda_i} \subset E_{\lambda_0}^\bot=H_\zeta^\bot$, for all $i=1,\dots,\ell$.
Thus $\oplus_{i=1}^\ell E_{\lambda_i} \subset H^\bot_\zeta$.
But we have
$$
\dim \oplus_{i=1}^\ell E_{\lambda_i} = \dim H^1(\hat{M},\C)-\dim H_\zeta=\dim H^\bot_\zeta.
$$
Therefore $\oplus_{i=1}^\ell E_{\lambda_i} =H^\bot_\zeta$, which means that
$$
H^1(\hat{M},\C)=H_\zeta \oplus H^\bot_\zeta.
$$
Hence the restriction of $(.,.)$ on $H_\zeta$ is non-degenerate.
\end{proof}

Let $\vartheta$ denote the imaginary part of $(.,.)$. By definition, $\vartheta$ is a real $2$-form on $H^1(\hat{M},\C)$. As an immediate consequence of Lemma~\ref{lm:inter:form:non:degen}, we get
\begin{Corollary}\label{cor:inter:form:symp}
The restriction of $\vartheta$ to $H_\zeta$ is non-degenerate. Hence $\vartheta^{\dim_\C H_\zeta}$ is a volume form on $H_\zeta$.
\end{Corollary}

\subsection{Definition of the volume form $d\vol$}\label{sec:def:vol:form}
Recall that  $\dim\strate=\dim V_\zeta=N$. Let $r$ be the number of indices $i\in \{1,\dots,n\}$ such that $d$ divides $k_i$. We assume that $k_i\in d\Z$ if and only if $i\in\{1,\dots,r\}$.

In what follows, given an element $c\in H_1(\hat{M},\hat{\Sig},\Z)$, we will consider $c$ as an element of $H^1(\hat{M},\hat{\Sig},\C)^*$. Denote by $c^\zeta$ the restriction of $c$ to $V_\zeta$.
\begin{Lemma}\label{lm:vol:form:Vzeta}
Let $\hat{s}_i$ and $c_i$, $i=1,\dots,r$, be as in Proposition~\ref{prop:ker:project:coh}. Then the $(N,N)$-form
 $$
 \Theta_\zeta:=\frac{1}{(N-r)!}(\frac{\imath}{2})^r \pp^*\vartheta^{N-r} \wedge c^\zeta_1\wedge\bar{c}^\zeta_1\wedge\dots\wedge c^\zeta_r\wedge\bar{c}^\zeta_r
 $$
 is a volume form on $V_\zeta$ which does not depend on the choices of $\hat{s}_{i}$ and $c_i$.
\end{Lemma}
\begin{proof}
By Lemma~\ref{lm:proj:coh:surj} and Proposition~\ref{prop:ker:project:coh}, we have
$$
\dim_\C H_\zeta=\dim_\C\pp(V_\zeta)=N-r.
$$
Therefore, there are $(N-r)$ cycles $b_1,\dots,b_{N-r} \in H_1(\hat{M},\Z)$  such that the map $H_\zeta \ra \C^{N-r}, \eta \mapsto (\eta(b_1),\dots,\eta(b_{N-r}))$ is an isomorphism.
This means that $\{b^\zeta_1,\dots,b^\zeta_{N-r}\}$ form  a basis of $H_\zeta^*$.
Hence for any $c\in H_1(\hat{M},\Z)$, $c^\zeta$ is a linear combination of $b^\zeta_1,\dots,b^\zeta_{N-r}$.
Note that since $\{b^\zeta_1,\dots,b^\zeta_{N-r}\}$ are independent in $H_\zeta^*$, they are also independent in $V_\zeta^*$.

By Corollary~\ref{cor:inter:form:symp}, $\vartheta^{N-r}$ is a volume form on $H_\zeta$. Therefore, there is a nonzero constant $\lambda \in \R^*$ such that
$$
\pp^*\vartheta^{N-r}=\lambda\cdot(\frac{\imath}{2})^{N-r} b^\zeta_1\wedge\bar{b}^\zeta_1\wedge\dots\wedge b^\zeta_{N-r}\wedge\bar{b}^\zeta_{N-r}.
$$
By Proposition~\ref{prop:ker:project:coh}, we know that there is a basis $(\eta_1,\dots,\eta_r)$ of $\ker\pp \cap V_\zeta$ such that $\eta_j(c_i)=(\zeta-1)\delta_{ij}$. This means that $(c^\zeta_1,\dots,c^\zeta_r)$ is independent on $V_\zeta^*$. Moreover, $c^\zeta_i$ does not belong to $\mathrm{Span}(b_1^\zeta,\dots,b^\zeta_{N-r})$ since we have $\eta_i(c_i)\neq 0$, while $\eta_i(b_1)=\dots=\eta_i(b_{N-r})=0$.
Therefore $(b_1^\zeta,\dots,b^\zeta_{N-r},c^\zeta_1,\dots,c_r^\zeta)$ is a basis of $V_\zeta^*$. Thus we have

\begin{eqnarray*}
% \nonumber % Remove numbering (before each equation)
  \Theta_\zeta &=& \frac{1}{(N-r)!}(\frac{\imath}{2})^r \pp^*\vartheta^{N-r} \wedge c^\zeta_1\wedge\bar{c}^\zeta_1\wedge\dots\wedge c^\zeta_r\wedge\bar{c}^\zeta_r\\
       &=& \frac{\lambda }{(N-r)!}(\frac{\imath}{2})^N b^\zeta_1\wedge\bar{b}^\zeta_1\wedge\dots\wedge b^\zeta_{N-r}\wedge \bar{b}^\zeta_{N-r} \wedge c^\zeta_1\wedge\bar{c}^\zeta_1\wedge\dots\wedge c^\zeta_r\wedge\bar{c}^\zeta_r.
\end{eqnarray*}
In particular, $\Theta_\zeta$ is a volume form on $V_\zeta$.

We now show that $\Theta_\zeta$ does not depend on the choice of $c_i$.
Let $c'_i$ be another path from $\hat{s}_{i}$ to $T_\zeta(\hat{s}_{i})$.
Then we can write $c_i=c'_i+a_i$, for some $a_i\in H_1(\hat{M},\Z)$.
Restricting to $V_\zeta$ gives $c_i^\zeta={c'}^\zeta_i+a^\zeta_i$.
Since $a^\zeta_i$ must be a linear combination of $(b_1^\zeta,\dots,b_{N-r}^\zeta)$, we get

\begin{eqnarray*}
% \nonumber % Remove numbering (before each equation)
  \Theta_\zeta &=& \frac{1}{(N-r)!}(\frac{\imath}{2})^r  \pp^*\vartheta^{N-r}\wedge({c'_i}^\zeta+a^\zeta_i)\wedge (\bar{c}'_i{}^{\zeta}+\bar{a}^\zeta_i)\wedge \left(\bigwedge_{j\neq i} c^\zeta_j\wedge \bar{c}^\zeta_j\right)  \\
   &=& \frac{1}{(N-r)!}(\frac{\imath}{2})^r \pp^*\vartheta^{N-r} \wedge {c'_i}^\zeta\wedge\bar{c}'_i{}^\zeta\wedge \left(\bigwedge_{j\neq i} c^\zeta_j\wedge \bar{c}^\zeta_j\right)  .
\end{eqnarray*}
Finally, if we replace $\hat{s}_{i}$ by another point $\hat{s}'_{i}$ in $\pi^{-1}(s_i)$, then there exists $k \in \{0,\dots, d-1\}$ such that $\hat{s}'_{i}=T_\zeta^k(\hat{s}_{i})$.
It follows that $c'_i:=T_\zeta^k(c_i)$ is a path from $\hat{s}'_{i}$ to $T_\zeta(\hat{s}'_{i})$.
Observe that ${c'_i}^\zeta=\zeta^k c^\zeta_i$.
Thus
$$
{c'_i}^\zeta\wedge \bar{c}'_i{}^\zeta=c^\zeta_i\wedge \bar{c}^\zeta_i,
$$
and the lemma follows.
\end{proof}

\begin{Proposition}\label{prop:vol:form:def}
For any $d>1$, the $(N,N)$-form $\Theta:=\frac{\Theta_\zeta}{|1-\zeta|^{2r}}$ on $V_\zeta$ gives rise to a well defined volume form $d\vol$ on $\strate$ that is parallel with respect to the affine manifold structure.
Moreover, $d\vol$ does not depend on the choice of $\zeta$.
\end{Proposition}
\begin{proof}
Any point $(X_0,q_0)$ in $\strate$ has a neighborhood that can be identified with an open subset of $V_\zeta$ by the local chart defined in Proposition~\ref{prop:eigen:sp:loc:coord}.
Since $\Theta$ is a volume form on $V_\zeta$, by Lemma~\ref{lm:vol:form:Vzeta}, it induces a volume form $d\vol$ on a neighborhood of $(X_0,q_0)$. It remains to show that $d\vol$ is invariant under the coordinate  changes of $\strate$.

Recall that the transition maps of local charts defined by the maps $\Xi$ (see Proposition~\ref{prop:eigen:sp:loc:coord}) arise from homotopy classes of homeomorphisms of the pair $(M,\Sig)$.
Fix a homeomorphism $f_0: (M,\Sig) \ra (X_0,Z(q_0))$.
Let $h: (M,\Sig) \ra (M,\Sig)$ be a homeomorphism of $M$ that is identity on the set $\Sig$. We can suppose that $h$ fixes the base point of $\pi_1(M',*)$.
Let $f'_0:=f_0\circ h^{-1}: (M,\Sig) \ra (X_0,Z(q_0))$, and $\Gamma'=h_*(\Gamma)=\ker(\veps_0\circ f'_{0*})$.

Let $\pi': (\hat{N},\hat{\Pi}) \ra (M,\Sig)$ be the cyclic coverings associated with $\Gamma'$, and $T'_\zeta: (\hat{N},\hat{\Pi}) \ra (\hat{N},\hat{\Pi})$ be the covering automorphism of $\pi'$ associated with $\zeta$.
By construction, $h$ lifts to a homeomorphism $\hat{h}:  (\hat{M},\hat{\Sig}) \ra (\hat{N},\hat{\Pi})$ such that $h\circ\pi =\pi'\circ\hat{h}$, and  $T'_\zeta=\hat{h}\circ T_\zeta \circ \hat{h}^{-1}$.
It follows that $\hat{h}^*$ restricts to an isomorphism from $V'_\zeta=\ker(T'_\zeta -\zeta\Id) \subset H^1(\hat{N},\hat{\Pi},\C)$ onto $V_\zeta$.

Let $\vartheta'$  denote the symplectic forms on $H^1(\hat{N},\C)$ which is induced by the intersection form on $H_1(\hat{N},\Z)$.
Since the intersection forms on  $H_1(\hat{N},\Z)$ and $H_1(\hat{M},\Z)$ are equivariant under $\hat{h}_*$, we have $\hat{h}_*\vartheta=\vartheta'$.
Note that $\hat{h}^*$ commutes with the projection $\pp$ (since $\hat{h}_*$ sends $H_1(\hat{M},\Z)$ bijectively onto $H_1(\hat{N},\Z)$). Therefore, we have
$$
\hat{h}_*(\pp^*{\vartheta}^{N-r})=\pp^*{\vartheta'}^{N-r}.
$$
Let ${\hat{s}_{i}}':=\hat{h}(\hat{s}_{i})$, for $i=1,\dots,r$.
%Recall that $\hat{h}$ maps ${\pi'}^{-1}(s_i)$ onto $\pi^{-1}(s_i)$. Thus, there is $k\in \N$ such that $\hat{s}''_i=T_\zeta^k(\hat{s}_i)$.
Since $T'_\zeta=\hat{h}\circ T_\zeta \circ \hat{h}^{-1}$, $c'_i:=\hat{h}(c_i)$ is a path from ${\hat{s}_{i}}'$ to $T'_\zeta({\hat{s}_{i}}')$.
It is straightforward to check that $\hat{h}_*\Theta_\zeta=\Theta'_\zeta$.  Therefore, $\Theta:=\frac{\Theta_\zeta}{|1-\zeta|^{2r}}$ gives a well defined volume form $d\vol$ on $\strate$.
Since $d\vol$ is given by constant volume forms in the local charts by $\Xi$, it is  parallel with respect to the  affine complex orbifold structure of $\strate$.

\medskip

It remains to show that $d\vol$ is independent of the choice of the primitive $d$-th root $\zeta$ of unity.
Let $\zeta'$ be another primitive $d$-th root of unity. There exists $k\in \N$ such that $\zeta'=\zeta^k$.
The covering automorphism of $\pi: (\hat{M},\hat{\Sig}) \ra (M,\Sig)$ associated with $\zeta'$ is $T_{\zeta'}:=T^k_\zeta$, and the $\zeta'$-eigenspace of $T_{\zeta'}$ in $H^1(\hat{M},\hat{\Sig},\C)$ is precisely $V_\zeta$.

Let $c'_i$ be a path from $\hat{s}_{i}$ to $T_{\zeta'}(\hat{s}_{i})=T^k_\zeta(\hat{s}_{i})$.
The volume form associated with $\zeta'$ is then defined by
$$
\Theta_{\zeta'}=\frac{1}{(N-r)!}(\frac{\imath}{2})^r \pp^*\vartheta^{N-r}\wedge \left(\bigwedge_{i=1}^r {c'_i}^{\zeta}\wedge \bar{c}'_i{}^{\zeta}\right).
$$
Since  $T_\zeta^{k-1}(c_i)*\dots*T_\zeta(c_i)*c_i$ is also a path from $\hat{s}_i$ to $T^k_\zeta(\hat{s}_i)$, there exists $a\in H_1(\hat{M},\Z)$ such that
$$
{c'_i}^{\zeta}=(1+\dots+\zeta^{k-1}) c_i^\zeta + a^\zeta=\frac{1-\zeta'}{1-\zeta}c^\zeta+a^\zeta.
$$
Therefore,
\begin{eqnarray*}
\Theta_{\zeta'} & = & \frac{1}{(N-r)!}\frac{|1-\zeta'|^{2r}}{|1-\zeta|^{2r}} (\frac{\imath}{2})^r \pp^*\vartheta^{N-r}\wedge \left( \bigwedge_{i=1}^r c_i^\zeta\wedge \bar{c}_i^\zeta\right)\\
       & = & \frac{|1-{\zeta'}|^{2r}}{|1-\zeta|^{2r}} \Theta_\zeta
\end{eqnarray*}
that is
$$
\frac{\Theta_{\zeta'}}{|1-\zeta'|^{2r}}=\frac{\Theta_{\zeta}}{|1-\zeta|^{2r}}.
$$
This implies that the volume form $d\vol$  does not depend on the choice of $\zeta$.
\end{proof}

\begin{Remark}\label{rk:def:vol:ex:seq}
Having in mind the exact sequence \eqref{eq:coh:ex:seq:eig:sp:red}, we have an alternative way to define the volume form $\Theta$ as follows: we define a Hermitian metric on $H^0(\hSig,\C)_\zeta$ by declaring the family $\{\eta_1,\dots,\eta_r\}$ in Proposition~\ref{prop:ker:project:coh} is an orthonormal basis with respect to this metric. We remark that this metric does not depend on the choice of $\hat{s}_i$ in $\pi^{-1}(\{s_i\})$. This is because a different choice of $\hat{s}_i$ results in multiplying $\eta_i$ by a $d$-th root of unity.
By the same reason, it does  not depend on the choice of $\zeta$ either.
Let $\theta$ be the volume form on $H^0(\hSig,\C)_\zeta$ associated to this metric.
The volume forms $\theta$ on $H^0(\hSig,\C)_\zeta$ and $\frac{1}{(N-r)!}\vartheta^{N-r}$ on $H_\zeta$ then induce a volume form on $V_\zeta$ via the exact sequence \eqref{eq:coh:ex:seq:eig:sp:red}.
It is straightforward to check that this volume form coincides with $\Theta$.
\footnote{The author thanks the anonymous referee for suggesting this definition.}
\end{Remark}

\subsection{Volume form in the case Abelian differentials}\label{sec:vol:form:d:1}
In the case $d=1$, we have $(\hat{M},\hat{\Sig})=(M,\Sig)$, $V_\zeta=H^1(M,\Sig,\C), H_\zeta = H^1(M,\C)$. For $i=1,\dots,n-1$, let $c_i$ be a path from $s_i$ to $s_n$. Then $(c_1,\dots,c_{n-1})$ is an independent family  in $H_1(M,\Sig,\C) \simeq H^1(M,\Sig,C)^*$. The following proposition follows from the same arguments as Proposition~\ref{prop:vol:form:def}.
\begin{Proposition}\label{prop:vol:form:abel:diff}
 For $d=1$, the form
 $$
 \Theta=\frac{1}{(2g)!}(\frac{\imath}{2})^{n-1}\pp^*\vartheta^{2g} \wedge \left(\bigwedge_{i=1}^{n-1}c_i\wedge \bar{c}_i\right)
 $$
 on $H^1(M,\Sig,\C)$ gives rise to a well defined volume form $d\vol$ on $\Omega \Mcal_{g,n}(\kappa)$.
\end{Proposition}

\subsection{Volume form on the projectivization}\label{sec:vol:form:project}
We now give the definition of the volume form $d\vol_1$ on $\Pb\strate$. Let $V_\zeta^+$ denote the set $\{\eta \in V_\zeta, \; (\eta,\eta) >0\}$. If $\eta=\Xi(X,q)$ for some $(X,q) \in \strate$, where $\Xi$ is the map defined in Proposition~\ref{prop:eigen:sp:loc:coord}, then we have
$$
(\eta,\eta)=\frac{\imath}{2}\int_{\hat{X}}\hat{\omega}\wedge\ol{\hat{\omega}} >0.
$$
Thus $\eta\in V_\zeta^+$. Let $\proj: V_\zeta \ra \Pb V_\zeta$ be the projectivization map of $V_\zeta$, and $\Pb V_\zeta^+$ be the image of $V_\zeta^+$ under $\proj$ .
By definition, $\proj\circ \Xi$ maps an open neighborhood of $(X,q)$ onto an open subset of $\Pb V_\zeta^+$. Consequently, $\Xi$ induces a biholomorphic map $\hat{\Xi}$ from a neighborhood of $\C^*\cdot(X,q)$ in $\Pb\strate$ onto an open subset of $\Pb V_\zeta^+$. We will use $\hat{\Xi}$  as local charts for $\Pb\strate$.

%Recall that $\Pb\strate$ can also be  identified with $\pstrate$. Therefore, the maps $\hat{\Xi}$ are also local charts of $\pstrate$.

Let $\mu_\Theta$ denote the measure on $V_\zeta$ which is defined by $\Theta$. Namely, $\mu_\Theta(U)=\int_U\Theta$, for all open subset $U$ of $V_\zeta$.
The measure $\mu_\Theta$ induces a measure $\mu^1_\Theta$ on $\Pb V_\zeta^+$ as follows: given an open subset $B$ of $\Pb V_\zeta^+$,
let $C(B)$ be the cone above $B$ in $V_\zeta$, that is $C(B)=\proj^{-1}(B)$.
Let
$$
C_1(B)=\{\eta\in C(B), \; 0< (\eta,\eta) \leq 1\} \subset V^+_\zeta.
$$
We then define
\begin{equation}\label{eq:def:vol:proj}
\mu^1_\Theta(B):=\frac{1}{d}\cdot \mu_\Theta(C_1(B)).
\end{equation}
The factor $\frac{1}{d}$ is introduced to take into account the fact for each $(X,q) \in \strate$, there are $d$ holomorphic $1$-forms $\hat{\omega}$ on $\hat{X}$ such that  $\varpi^*q=\hat{\omega}^d$.
By a direct computation, one can see that $\mu^1_\Theta$ is actually given by a volume form $\Theta_1$ on $\Pb V^+_\zeta$.
Since $\Pb\strate$ are locally identified with $\Pb V_\zeta^+$, $\Theta_1$ provides us with a volume form $d\vol_1$ on $\Pb \strate$.

\subsection{Comparison with Masur-Veech volumes}\label{sec:vol:form:compare}
For $d\in \{1,2,3,4,6\}$, there exists another natural volume form on $\strate$  that we now  describe.
Recall that if $A$ is a $\Z$-module, then $H^1(\hat{M},\hat{\Sig},A)$ is the space of morphisms of $\Z$-modules $\eta: H_1(\hat{M},\hat{\Sig},\Z) \ra A$.
Let
$$
\Lambda_\zeta= \left\{
\begin{array}{ll}
 V_\zeta\cap H^1(\hat{M},\hat{\Sig},\Z\oplus \imath\Z) & \text{ if } d\in \{1,2,4\},\\
 V_\zeta\cap H^1(\hat{M},\hat{\Sig},\Z[\zeta]) & \text{ if } d\in \{3,6\}.
\end{array}
\right.
$$
Since $V_\zeta$ is defined over $\Q(\zeta)$, $\Lambda_\zeta$ is a lattice of $V_\zeta$. There is unique volume form on $V_\zeta$ proportional to the Lebesgue measure such that the co-volume of $\Lambda_\zeta$ is $1$.
Since the transition maps of the local charts by period mappings preserve $\Lambda_\zeta$, this  volume form gives a well defined volume form on $\strate$, that will be referred to as the {\em Masur-Veech measure} and denoted by $d\vol^*$.  Consequently, one can define a volume form $d\vol^*_1$ on $\pstrate$ in the same way as $d\vol_1$.
Note that our normalization for $d\vol^*$ differs slightly from the normalizations in \cite{AEZ_ann_ENS} or \cite{Engel-I}.
The following proposition follows immediately from the definition of $d\vol$ and $d\vol^*$.
\begin{Proposition}\label{prop:ratio:v:forms:const}
For each stratum $\strate$ with $d\in \{1,2,3,4,6\}$, there is a real constant $\lambda$ such that $d\vol=\lambda d\vol^*$.
\end{Proposition}

In the remainder of this section, we investigate the possible values of the constant $\lambda$.

\begin{Proposition}\label{prop:compare:v:form:d:1}
For any stratum $\strateabdiff$ of Abelian differentials in genus $g$, we have
\begin{equation*}
\frac{d\vol}{d\vol^*}=\frac{(-1)^g}{2^{2g}}.
\end{equation*}
\end{Proposition}
\begin{proof}
In this case $(\hat{M},\hat{\Sig})=(M,\Sig)$, $V_\zeta=H^1(M,\Sig,\C)$, $N=2g+n-1$, and $\Lambda_\zeta=H^1(M,\Sig,\Z\oplus\imath\Z)$. Fix a symplectic basis $\{a_1,\dots,a_g,b_1,\dots,b_g\}$ of $H_1(M,\Z)$, and let $c_1,\dots,c_{n-1}$ be a family of paths joining $s_n$ to $s_1,\dots,s_{n-1}$ respectively. We identify $H^1(M,\Sig,\C)$ with $\C^{2g+n-1}$ by the mapping $\eta \mapsto (\eta(a_1),\dots,\eta(a_g),\eta(b_1),\dots,\eta(b_g),\eta(c_1),\dots,\eta(c_{n-1}))$. Let $(z_1,\dots,z_{2g+n-1})$ be the canonical complex coordinates of $\C^{2g+n-1}$. We will write $z_j=x_j+\imath y_j$, with $x_j,y_j \in \R$. We then have
$$
d\vol^*=dx_1dy_1\dots dx_{2g+n-1}dy_{2g+n-1}.
$$
In the coordinates $(z_1,\dots,z_{2g+n-1})$, the intersection form $(.,.)$ is given by
\begin{equation}\label{eq:herm:form:inter}
\Hb:=\frac{\imath}{2}\sum_{j=1}^g (dz_j\otimes d\bar{z}_{g+j} - dz_{g+j}\otimes d\bar{z}_j).
\end{equation}
It follows that
\begin{equation}\label{eq:sympl:form:inter}
\vartheta=\mathrm{Im}\Hb=\frac{1}{4}\sum_{j=1}^g (dz_j\wedge d\bar{z}_{g+j}-dz_{g+j}\wedge d\bar{z}_j).
\end{equation}
Hence
\begin{eqnarray*}
\frac{1}{(2g)!}(\frac{\imath}{2})^{n-1}\vartheta^{2g} \wedge \left(\bigwedge_{i=2g+1}^{2g+n-1}dz_i\wedge d\bar{z}_i\right)& = &  \frac{1}{4^{2g}} (\frac{\imath}{2})^{n-1}  \left(\bigwedge_{j=1}^g dz_jd\bar{z}_jdz_{g+j}d\bar{z}_{g+j}\right)\wedge \left(\bigwedge_{i=2g+1}^{2g+n-1}dz_id\bar{z}_i\right)\\
                             & = &  \frac{1}{4^{2g}} 2^{2g}(-1)^g dx_1dy_1\dots dx_{2g+n-1}dy_{2g+n-1}\\
                             & = & \frac{(-1)^g}{2^{2g}}d\vol^*.
\end{eqnarray*}
Thus we have $\displaystyle \frac{d\vol}{d\vol^*}=\frac{(-1)^g}{2^{2g}}$.
\end{proof}

\begin{Proposition}\label{prop:vol:const:rat}
%Let $N=\dim_\C\strate$.
For $d\in \{2,3,4,6\}$, we have
\begin{itemize}
\item if $d\in \{2,4\}$ then $\displaystyle \frac{d\vol}{d\vol^*} \in \Q$,

\item if $d \in \{3,6\}$ then $\displaystyle \frac{d\vol}{d\vol^*} \in (\sqrt{3})^r\cdot\Q$
\end{itemize}
\end{Proposition}
\begin{proof}
Let $\{a_1,\dots,a_{\hat{g}},b_1,\dots,b_{\hat{g}}\}$ be a symplectic basis of $H_1(\hat{M},\Z)$.
For $i=1,\dots,r$, let $c_i$ be as in Proposition~\ref{prop:ker:project:coh}.
Note that $\{c_1,\dots,c_r\}$ are independent in $H_1(\hat{M},\hat{\Sig},\C)$, and
$$
\mathrm{Span}(a_1,\dots,a_{\hat{g}},b_1,\dots,b_{\hat{g}})\cap \mathrm{Span}(c_1,\dots,c_r)=\{0\}.
$$
We can complete the family $\{a_1,\dots,a_{\hat{g}},b_1,\dots,b_{\hat{g}},c_1,\dots,c_r\}$ with some cycles $c_{r+1},\dots,c_{\hat{n}-1}$ such that
$$
\Bcal:=\{a_1,\dots,a_{\hat{g}},b_1,\dots,b_{\hat{g}},c_1,\dots,c_{\hat{n}-1}\}
$$
is a basis of $H_1(\hat{M},\hat{\Sig},\Z)$. We will use $\Bcal$ to identify $H^1(\hat{M},\hat{\Sig},\C)$ with $\C^{2\hat{g}+\hat{n}-1}\simeq \C^{2\hat{g}}\times\C^{\hat{n}-1}$. In this setting, the projection $\pp: H^1(\hat{M},\hat{\Sig},\C) \ra H^1(\hat{M},\C)$ is  given by the natural projection $\pp: \C^{2\hat{g}+\hat{n}-1}\simeq \C^{2\hat{g}}\times \C^{\hat{n}-1} \ra \C^{2\hat{g}}$.

By definition, $V_\zeta \subset \C^{2\hat{g}+\hat{n}-1}$ is the eigenspace for the eigenvalue $\zeta$ of $T_\zeta$.
Since $T_\zeta$ is given by an integral matrix, there is a basis of $V_\zeta$ consisting of vectors with coordinates in $\Q(\zeta)$.
By Proposition~\ref{prop:ker:project:coh}, the projection $H_\zeta=\pp(V_\zeta)$ of $V_\zeta$ is a subspace of dimension $K:=(N-r)$ of  $\C^{2\hat{g}}$. Thus there exists a family of $K$ indices $\{i_1,\dots,i_K\} \subset \{1,\dots, 2\hat{g}\}$ such that the map
$$
\begin{array}{cccc}
\qq_1: &\C^{2\hat{g}} & \ra &  \C^K\\
       & (z_1,\dots,z_{2\hat{g}}) & \mapsto & (z_{i_1},\dots,z_{i_K})
\end{array}
$$
restricts to an isomorphism from $H_\zeta$ to $\C^K$.  %We will also consider $\qq_1$ as a projection from $\C^{2\hat{g}+\hat{n}-1}$ onto $\C^K$.
Define
$$
\begin{array}{cccc}
\qq: & \C^{2\hat{g}+\hat{n}-1} & \ra  & \C^N\\
     & (z_1,\dots,z_{2\hat{g}+\hat{n}-1}) & \mapsto & (z_{i_1},\dots,z_{i_K},z_{2\hat{g}+1},\dots,z_{2\hat{g}+r})
\end{array}
$$
To simplify the notation, for $z=(z_1,\dots,z_{2\hat{g}+\hat{n}-1})$, we will write $\qq(z)=(\qq_1\circ \pp(z),\qq_2(z))$, where
$$
\qq_1\circ\pp(z)=(z_{i_1},\dots,z_{i_K}) \in \C^K \text{ and } \qq_2(z)=(z_{2\hat{g}+1},\dots,z_{2\hat{g}+r})\in \C^{r}.
$$
\begin{Claim}\label{clm:proj:bijective}
The map $\qq$ restricts to an isomorphism from $V_\zeta$ onto $\C^{N}$.
\end{Claim}
\begin{proof}
Since $\dim V_\zeta=N$, it is enough to show that $\qq_{|V_\zeta}$ is injective. Let $z \in \ker\qq\cap V_\zeta$.
Since $\qq_1\circ\pp(z)=0$, we get $\pp(z) \in \ker\qq_1\cap H_\zeta$. But the restriction of $\qq_1$ to $H_\zeta$ is a bijection, therefore $\pp(z)=0\in \C^{2\hat{g}}$, which means that $z\in \ker\pp \cap V_\zeta$.

By Proposition~\ref{prop:ker:project:coh}, $\qq_2$ restricts to a bijection from $\ker\pp \cap V_\zeta$ onto $\C^r$. We have $\qq_2(z)=0$, therefore, $z=0$. We can then conclude that $\qq$ is injective, and hence bijective.
\end{proof}
Claim~\ref{clm:proj:bijective} implies that there is a linear map $\rr: \C^N \ra \C^{2\hat{g}+\hat{n}-1}$ such that  $\rr(\C^N)=V_\zeta$ and  $\qq\circ \rr =\id_{\C^N}$. Note that $\rr$ is given by a matrix with coefficients in $\Q(\zeta)$.
Given $w=(w_1,\dots,w_N)\in \C^{N}$, let $z=(z_1,\dots,z_{2\hat{g}+\hat{n}-1})=\rr(w)$. By definition, we have
$$
z_{i_j}=w_j, \; j=1\dots,K=N-r, \text{ and } z_{2\hat{g}+j}=w_{K+j}, \; j=1,\dots,r.
$$

\begin{Claim}\label{clm:ab:per:function}
For $i\in \{1,\dots,2\hat{g}\}$, the coordinate $z_i$ is a linear combination of $(w_1,\dots,w_K)$.
\end{Claim}
\begin{proof}
Assume that $w_1=\dots=w_K=0$. We have $w=\qq\circ\rr(w)=\qq(z)=(\qq_1\circ\pp(z),\qq_2(z))$, which implies that $\qq_1\circ\pp(z)=0$. It follows that $\pp(z)\in \ker\qq_1\cap \pp(V_\zeta)=\ker\qq_1\cap H_\zeta$. But the restriction of $\qq_1$ to $H_\zeta$ is a bijection, therefore $\pp(z)=0$, that is $z_1=\dots=z_{2\hat{g}}=0$. This means that as linear form on $\C^N$, $\ker(z_i)$ contains $\cap_{j=1}^K\ker(w_j)$. Hence $z_i$ is a linear combination of $(w_1,\dots,w_K)$.
\end{proof}
Since $\rr$ is defined over $\Q(\zeta)$ and $\vartheta$ is given by \eqref{eq:sympl:form:inter} we have
$$
\rr^*\vartheta=\sum_{i,j=1}^{K}\vartheta_{ij}dw_i\wedge d\bar{w}_j
$$
where $\vartheta_{ij} \in \Q(\zeta)$ and $\vartheta_{ji}=-\bar{\vartheta}_{ij}$.

\begin{Claim}\label{clm:det:inters:form}
We have
$$
\det(\vartheta_{ij}) \in \Q(\zeta)\cap \imath^K\R.
$$
\end{Claim}
\begin{proof}
Since $(\vartheta_{ji})=(-\bar{\vartheta}_{ij})$, we have $\det(\vartheta_{ij})=(-1)^K \overline{\det(\vartheta_{ij})}$,
which means that $\det(\vartheta_{ij}) \in \imath^K\R$.
By definition, the coefficients $\vartheta_{ij}$ belong to $\Q(\zeta)$, therefore  $\det(\vartheta_{ij}) \in \Q(\zeta)\cap \imath^K\R$.
\end{proof}

\begin{Claim}\label{clm:rat:v:forms:loc}
We have
\begin{equation*}
\frac{\rr^*d\vol}{\rr^*d\vol^*}=\left\{
\begin{array}{ll}
\frac{(-2\imath)^K}{|1-\zeta|^{2r}}\det(\vartheta_{ij})\ell & \text{ if } d\in \{2,4\}\\
\frac{(-2\imath)^K}{|1-\zeta|^{2r}}\det(\vartheta_{ij})(\frac{\sqrt{3}}{2})^N\ell  & \text{ if } d\in \{3,6\}.
\end{array}
\right.
\end{equation*}
where $\ell$ is a positive integer.
\end{Claim}

\begin{proof}
By definition,
\begin{eqnarray*}
% \nonumber % Remove numbering (before each equation)
   \rr^*d\vol & = & \frac{1}{|1-\zeta|^{2r}}\cdot\frac{1}{K!}(\frac{\imath}{2})^r (\rr^*\vartheta)^K\wedge \left(\bigwedge_{j=K+1}^N dw_j\wedge d\bar{w}_j \right)  \\
     & = & \frac{1}{|1-\zeta|^{2r}}\cdot(\frac{\imath}{2})^r\det(\vartheta_{ij})dw_1d\bar{w}_1\dots dw_Nd\bar{w}_N\\
     %&= & \chi dw_1d(\imath)^N\bar{w}_1\dots dw_Nd\bar{w}_N.
\end{eqnarray*}
Writing $w_i=u_i+\imath v_i$, we get
$$
\rr^*d\vol=\frac{(-2\imath)^K}{|1-\zeta|^{2r}}\det(\vartheta_{ij})du_1dv_1\dots du_Ndv_N.
$$
Let $\Lambda=(\Z\oplus\imath\Z)^N$ if $d\in\{2,4\}$, and $\Lambda= (\Z\oplus e^{\frac{2\pi\imath}{3}}\Z)^N$ if $d \in \{3,6\}$.
%Note that $\Lambda$ is a lattice of $\C^N$.
By the definition of $\qq$, we see that $\Lambda=\qq((\Z+\imath\Z)^{2\hat{g}+\hat{n}-1})$ if $d=2$,
and $\Lambda=\qq((\Z+\zeta\Z)^{2\hat{g}+\hat{n}-1})$ if $d\in \{3,4,6\}$.

Recall that $\Lambda_\zeta=(\Z\oplus\imath\Z)^{2\hat{g}+\hat{n}-1}\cap V_\zeta$ if $d=2$, and  $\Lambda_\zeta=(\Z\oplus\zeta\Z)^{2\hat{g}+\hat{n}-1}\cap V_\zeta$ if $d\in\{3,4,6\}$. Since $\Lambda_\zeta$ is a lattice of $V_\zeta$ and $\qq: V_\zeta \ra \C^N$ is an isomorphism, $\qq(\Lambda_\zeta)$ must be a lattice of $\C^N$ which is contained in $\Lambda$ .
Let $\ell$ be the index of $\qq(\Lambda_\zeta)$ in $\Lambda$.
By definition, $\Lambda_\zeta$ has covolume $1$ with respect to $d\vol^*$.
It follows that $\Lambda$ has covolume $\frac{1}{\ell}$ with respect to $\rr^*d\vol^*$.

With respect to the Lebesgue volume form $du_1dv_1\dots du_Ndv_N$ on $\C^N$, the covolume of the lattice $\Lambda$ is $1$ if $d\in\{2,4\}$, and $(\frac{\sqrt{3}}{2})^N$ if $d\in \{3,6\}$. Therefore
$$
\frac{\rr^*d\vol}{\rr^*d\vol^*}=\left\{
\begin{array}{ll}
\frac{(-2\imath)^K}{|1-\zeta|^{2r}}\det(\vartheta_{ij})\ell & \text{ if } d\in \{2,4\}\\
\frac{(-2\imath)^K}{|1-\zeta|^{2r}}\det(\vartheta_{ij})(\frac{\sqrt{3}}{2})^N\ell  & \text{ if } d\in \{3,6\}.
\end{array}
\right.
$$
\end{proof}

Because of Claim~\ref{clm:det:inters:form}, we have the following
\begin{itemize}
 \item If $d=2$, then $\Q(\zeta)=\Q$, and the condition $\det(\vartheta_{ij}) \in \imath^K\R$ implies that $K$ is even (note that $K=2g+n-2-r$, and $n-r$ is the number of odd order zeros which must be even). Hence $(-2\imath)^K\det(\vartheta_{ij})\ell \in \Q$.

 \item If $d=4$, then $\Q(\zeta)=\Q(\imath)$. Therefore $(-2\imath)^K\det(\vartheta_{ij}) \in \Q(\imath)\cap \R=\Q$, which implies that $(-2\imath)^K\det(\vartheta_{ij})\ell \in \Q$.

 \item If $d\in \{3,6\}$, then $\Q(\zeta)=\Q(e^{\frac{\imath\pi}{3}})$. We have two cases: if $K$ is odd then  $\Q(e^{\frac{\imath\pi}{3}}) \cap \imath\R=\imath\sqrt{3}\Q$, and if $K$ is even then $\Q(e^{\frac{\imath\pi}{3}}) \cap \R=\Q$. Thus
   $$
   (-2\imath)^K\det(\vartheta_{ij})(\frac{\sqrt{3}}{2})^N\ell = (\frac{\sqrt{3}}{2})^r(-\imath\sqrt{3})^K\det(\vartheta_{ij})\ell \in (\sqrt{3})^r\cdot\Q.
   $$
   %Since $\displaystyle (-\imath\sqrt{3})^K\det(\vartheta_{ij}) \in \Q$, the proposition follows.
 \end{itemize}
Since for all $d \in \{2,3,4,6\}$, $|1-\zeta|^2$ is always an integer, the proposition follows.
\end{proof}

%**********************************************************************
%**********************************************************************
%**********************************************************************
%**********************************************************************
\section{Delaunay triangulation}\label{sec:Delaunay}
In this section we review some basic properties of the Delaunay triangulation of flat surfaces. Our main reference on the matter is \cite[Sect. 4 and 5]{MS91} (see also \cite{Th98}).
Let $M$ be a flat surface with conical singularities and $\Sig$ a finite subset of $M$ which contains all the singularities. Let us denote by $\dd$ the distance induced by the flat metric on $M$.
We first describe the decomposition of $M$ into {\em Voronoi cells}.
\begin{Definition}\label{def:Voronoi:cells}
The $2$-dimensional Voronoi cells of the pair $(M,\Sig)$ are  connected components of the set of points in $M$ which have a unique length-minimizing path to $\Sig$. The $1$-dimensional Voronoi cells are connected components of the set of points that have exactly two length-minimizing paths to $\Sig$. Finally, the $0$-dimensional Voronoi cells are the points which have at least three length minimizing paths to $\Sig$.
\end{Definition}
The  Voronoi $0$-cells  are isolated points in $M$, in particular they are finite. The  Voronoi $1$-cells are geodesic segments with endpoints being $0$-cells. The $2$-cells are open domains bounded by the union of some $1$-cells and $0$-cells, each $2$-cell contains a unique point in $\Sig$.

The {\em Delaunay decomposition} is the dual of the Voronoi decomposition, which is defined as follows. For any  Voronoi $0$-cell $s \in M$, let $\dd_s=\dd(s,\Sig)$. There is a map $\phi_s: D(0,\dd_s) \ra M$, where $D(0,\dd_s)$ is the disc of radius $\dd_s$ centered at $0$ in the plane, which satisfies
\begin{itemize}
 \item $\phi_s$ is locally isometric,

 \item $\phi_s(0)=s$.
\end{itemize}
By definition, $\phi_s^{-1}(\Sig)$ is a finite subset of $\partial D(0,\dd_s)$ which has at least $3$ points. The convex hull of $\phi_s^{-1}(\Sig)$ is a convex polygon $H_s$ inscribed in $D(0,\dd_s)$. The restriction of $\phi_s$ into $H_s$ is an embedding, and if $s'$ is another  Voronoi $0$-cell then $\phi_{s'}(H_{s'})$ and $\phi_s(H_s)$ can only meet in their boundary (see \cite[Lem. 4.2]{MS91}).

\begin{Theorem}[\cite{MS91}]
The domains $\{\phi_s(H_s), \, s \hbox{ is a Voronoi $0$-cell of } (M,\Sig)\}$ define a decomposition of $M$ into cells:
\begin{itemize}
\item the $2$-cells of this decomposition are $\{\phi_s(\inter(H_s)), \, s  \hbox{ is a Voronoi $0$-cell of } (M,\Sig)\}$,

\item the $1$-cells are the geodesic segments (with their endpoints excluded) joining the points in $\Sig$ that are contained in the border of some $2$-cells,

\item the $0$-cells are points in $\Sig$.
\end{itemize}
This decomposition is called  {\em Delaunay decomposition} of $(M,\Sig)$.
\end{Theorem}

The duality between the Delaunay decomposition and the Voronoi decomposition can be seen as follows: by construction, it is clear that the set of $i$-cells of the Delaunay decomposition is in bijection with the set of $(2-i)$-cells of the Voronoi decomposition, for $i=0,2$. Let $\gamma$ be a Voronoi $1$-cell, and $s$ and $s'$ the Voronoi $0$-cells that are the endpoints of $\gamma$. Let $v$ be the holonomy vector of $\gamma$. By translating the disc $D(0,\dd_{s'})$ by  $v$, we can define a map $\phi: D(0,\dd_s)\cup D(v,\dd_{s'}) \ra M$ which is a local isometry, and sends $0$ to $s$ and $v$ to $s'$.
The circles $\partial D(0,\dd_s)$ and $\partial D(v,\dd_{s'})$ meet at two points which are mapped to points in $\Sig$. The image of the segment between these two points is the  Delaunay $1$-cell dual to $\gamma$.

\begin{Definition}\label{ef:Delaunay:triang}
 A triangulation of $M$ which is obtained   by subdividing the $2$-cells of the Delaunay decomposition of $(M,\Sig)$ into triangles is called a {\em Delaunay triangulation} of $(M,\Sig)$.
\end{Definition}
%Subdivide the Delaunay $2$-cells of $(M,\Sig)$ into triangles by geodesics joining the points in $\Sig$, we obtain a {\em Delaunay triangulation} of $(M,\Sig)$.
The Delaunay triangulations  are obviously not unique, but there are only finitely many of them.

\medskip

 Assume from now on that $M$ is a translation surface,  that is the metric on $M$ is defined by a holomorphic $1$-form. A {\em cylinder} on $(M,\Sig)$ is an open subset $C$ of $M\setminus\Sig$ which is isometric to $(\R/\ell\Z)\times(0,h)$, for some $\ell,h \in \R_+^*$, and not properly contained in a larger subset with the same properties. The parameters $\ell$ and $h$ are called the {\em circumference} and the {\em height} (sometimes {\em width}) of $C$ respectively.

By definition, there is an isometric embedding $\varphi: (\R/\ell\Z)\times (0,h) \ra M$ such that $\varphi(\R/\ell\Z)\times(0,h)=C$. We can extend $\varphi$ by continuity to a map $\ol{\varphi}: (\R/\Z\ell)\times[0,h] \ra M$. The images of $(\R/\ell\Z)\times\{0\}$ and $(\R/\ell\Z)\times\{h\}$ under $\ol{\varphi}$ are called the {\em boundary components} of $C$. Note that each boundary component must contain a point in $\Sig$  by definition, and the two boundary components are not necessarily disjoint in $M$.

A cylinder can be also defined as the union of all the simple closed geodesics in the same free homotopy class in $M\setminus\Sig$. Those simple closed geodesics are called the {\em core curves} of the cylinder. Each cylinder is uniquely determined by any of its core curve.

To any path $a$ in $M$ with endpoints in $\Sig$, the integration of the holomorphic $1$-form defining the flat metric structure along $a$ provides us with a complex number which will be called the {\em period} or the {\em holonomy vector} of $a$ (here, we identify $\C$ with $\R^2$).
A {\em saddle connection} on $(M,\Sig)$ is a geodesic segment with endpoints in $\Sig$ which contains no point in $\Sig$ in its interior. If $a$ is a saddle connection, its length will be denote by $|a|$. In this case, $|a|$ is also equal to the module of its period.

\medskip

Consider now a Delaunay triangulation $\Tcal$ of $(M,\Sig)$.  Recall that by definition, all the edges of $\Tcal$ are saddle connections.
The following result tells  us that if an edge of the Delaunay triangulation has sufficiently large length, then it must cross a cylinder whose height is greater than its circumference (see \cite[Th. 5.3 and Prop. 5.4]{MS91}).

\begin{Proposition}\label{prop:long:sc:high:cyl}
Assume that $\Aa(M)\leq 1$. Let $e$ be an edge of $\Tcal$. If $|e| > 2\sqrt{2/\pi}$ then $e$ must cross a cylinder $C$ whose height $h$ is greater than the circumference $\ell$. Moreover, we have
\begin{equation}\label{eq:long:sc:bound:1}
h \leq |e| \leq \sqrt{h^2+\ell^2} \leq \sqrt{2}h.
\end{equation}
\end{Proposition}

\begin{Definition}\label{def:long:cyl}
A cylinder whose height is greater than $\frac{2\sqrt{2}}{\sqrt{\pi}}\cdot \sqrt{\Aa(M)}$ will be called a {\em long cylinder}.
\end{Definition}
To ease the notation, in what follows we will write $\alpha=\frac{2\sqrt{2}}{\sqrt{\pi}}$.

\begin{Lemma}\label{lm:long:cyls:disjoint}
If $C$ and $C'$ are two long cylinders on $(M,\Sig)$, then the core curves of $C$ do not cross $C'$, that is $C$ and $C'$ are disjoint.
\end{Lemma}
\begin{proof}
Let $h$ and  $\ell$ (resp. $h'$ and $\ell'$)  denote the height and circumference of $C$ (resp. of $C'$) respectively.
We have $\Aa(C)=h\ell \leq \Aa(M)$. Therefore, $\ell \leq \frac{\Aa(M)}{h} < \frac{\sqrt{\Aa(M)}}{\alpha}$.
%Similarly, we also have $\ell' \leq < \frac{\sqrt{\Aa(M)}}{\alpha}$.

Let $c$ be a core curve of $C$.
If $c$ intersects $C'$, it must cross $C'$  entirely. Hence $|c|=\ell \geq h' >\alpha\sqrt{\Aa(M)}$, which implies
$$
\frac{\sqrt{\Aa(M)}}{\alpha} > \alpha\sqrt{\Aa(M)}.
$$
Since $\alpha >1$, this is impossible.
Hence  $C$ and $C'$ are disjoint.
\end{proof}

Conversely to Proposition~\ref{prop:long:sc:high:cyl}, we have

\begin{Lemma}\label{lm:sc:cross:long:cyl}
Let $C$ be a long cylinder in  $(M,\Sig)$ whose height and circumference are denoted by $h$ and $\ell$ respectively.
%Assume that $A:=\Aa(M)\leq 1$.
%Denote by $\ell$ the circumference of $C$.
Let $A=\Aa(M)$. Let $e$ be an edge of a Delaunay triangulation that crosses $C$.  Then
\begin{itemize}
\item[(i)] if $(x,y)$ are the coordinates of the period vector of $e$ in  the orthonormal basis $(v_1,v_2)$ of $\R^2$, where $v_1$ is the unit vector in the direction of the core curves of $C$,  then $|x| \leq \ell$,

\item[(ii)] $|e| < h+\frac{\sqrt{A}}{\alpha^3}$,

\item[(iii)] $e$ crosses $C$ once, and $C$ is the unique long cylinder that is crossed by $e$.

\end{itemize}
\end{Lemma}
\begin{proof}
Without loss of generality, we can suppose that $C$ is horizontal.
By construction, there is a Voronoi $0$-cell $s$ and an isometric embedding $\phi_s: D(0,R) \ra M$, where $R=\dd_s$, such that $\phi_s(0)=s$, and  $e$ is the image of a secant $\tilde{e}$ of the circle $\partial D(0,R)$ under $\phi_s$.
Since $|e|=|\tilde{e}|\leq 2R$, and $|e| \geq h$, we get that
\begin{equation}\label{ineq:long:cyl:R}
R > \frac{h}{2}.
\end{equation}

The preimage of the boundary of $C$ in $D(0,R)$ consists of two horizontal secants of the circle $\partial D(0,R)$ which will be denoted by $a'$ and $a''$. Since no point in the interior of $D(0,R)$ is mapped to a point in $\Sig$,  the lengths of $a'$ and $a''$ are smaller than the circumference $\ell$.
Note also that since $\Aa(C)=\ell h \leq \Aa(M)= A$, we must have
\begin{equation}\label{ineq:long:cyl:ell}
\ell \leq \frac{A}{h} < \frac{\sqrt{A}}{\alpha}.
\end{equation}
Therefore,
\begin{equation}\label{ineq:long:cyl:a}
\max\{|a'|,|a''|\} < \frac{\sqrt{A}}{\alpha}.
\end{equation}

Let $L'$ and $L''$ be the horizontal lines that contain $a'$ and $a''$ respectively.
Let $H$ be the horizontal band bounded by $L'$ and $L''$ in the plane.
%Since the domain $H\cap D(0,R)$ is mapped into $C$, no point in this domain is mapped to $\Sig$. This means that the endpoints of $\tilde{e}$ are contained in the two arcs of $\partial D(0,R)$ that are not in $H$.
Let $\sig'$ (resp. $\sig''$) be the arc of $\partial D(0,R)$ outside of $H$ which has the same endpoints as $a'$ (resp. as $a''$).
Since $\tilde{e}$ crosses $H$ entirely, it must have an endpoint in $\sig'$ and an endpoint in $\sig''$ (see Fig.~\ref{fig:edge:cross:cyl}).

\begin{figure}[htbp]
\centering
\begin{tikzpicture}[scale=0.5]
\fill[black!10] (-7,-3) -- (7,-3) -- (7,4) -- (-7,4) -- cycle;
 \draw (0,0) circle (5);
 \draw (-7,4) -- (7,4) (-7,-3) -- (7,-3);

 \filldraw[fill=black] (0,0)  circle (2pt);

 \draw[<->, >= angle 45] (0,0) -- (0,4); \draw (0,2)  node[right] {$\tiny h'$};
 \draw[<->, >= angle 45] (0,0) -- (0,-3); \draw (0,-1.2) node[right] {$\tiny h''$};
 %\draw (0,0) node[right] {$\tiny 0$};

 \draw[very thick] (115:5) -- (270:5);
 \filldraw[black] (115:5) circle (4pt);
 \filldraw[black] (270:5) circle (4pt);
 \draw (-2,1) node {$\tiny \tilde{e}$};

 \draw (1.5,4) node[below] {$\tiny a'$};
 \draw (2.5,-3) node[above] {$\tiny a''$};
 \draw (0,5) node[above] {$\tiny \sigma'$};
 \draw (2,-4) node {$\tiny \sigma''$};

 %\draw[dashed] (-6.5,4) -- (-6.5,-3) (5.5,4) -- (5.5,-3);
 \draw[<->, >=angle 45] (5.5,4) -- (5.5,-3);
 \draw (5.5,0.5) node[right] {$\tiny h$};
 \draw (-7,4) node[left] {$\tiny L'$};
 \draw (-7,-3) node[left] {$\tiny L''$};
 \draw (-7,0.5) node[left] {$\tiny H$};

 %\draw[<->, >=angle 45] (-6.5,-6) -- (5.5,-6);
 %\draw (-0.5,-6) node[below] {$\tiny \ell$};
\end{tikzpicture}
\caption{Edge crossing a cylinder.}
\label{fig:edge:cross:cyl}
\end{figure}
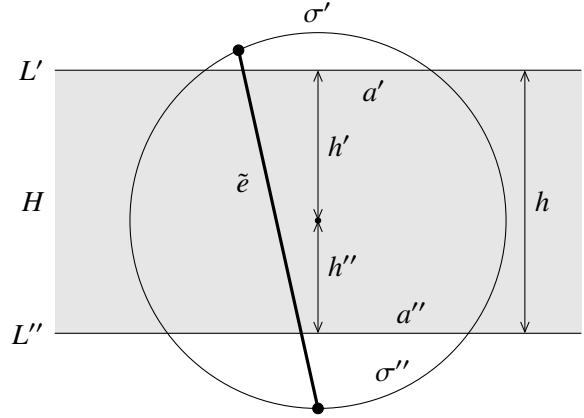
%Observe that $\max\{|a'|,|a''|\}\leq \ell$, because otherwise either $a'$ or $a''$ contains a point in its interior which gets mapped to a point in $\Sig$.
%But by construction, no point in the open disc $D(0,R)$ is mapped to a point in $\Sig$.

We  claim that the origin is contained in $H$.
Let $h'$ and $h''$ be the distance from the origin to $a'$ and to $a''$ respectively.
Observe that
$$
R^2-{h'}^2 = \frac{|a'|^2}{4} < \frac{A}{4\alpha^2},
$$
which implies
\begin{equation}\label{ineq:long:cyl:h1}
R-h' < \frac{A}{4\alpha^2(R+h')} < \frac{A}{4\alpha^2 R} < \frac{A}{2\alpha^2 h},
\end{equation}
(here we used \eqref{ineq:long:cyl:R}). By the same argument, we also get
\begin{equation}\label{ineq:long:cyl:h2}
R-h'' < \frac{A}{2\alpha^2 h}.
\end{equation}
By construction, $h$ is the distance between $L'$ and $L''$. If $H$ does not contain the origin, then
$$
h=|h'-h''|\leq \max\{R-h',R-h''\}< \frac{A}{2\alpha^2 h}
$$
which implies
$$
h < \frac{\sqrt{A}}{\sqrt{2}\alpha}.
$$
Since $\alpha > \frac{1}{\sqrt{2}\alpha}$, we get a contradiction to the hypothesis that $C$ is a long cylinder.
Thus $0$ must be contained in $H$, and we have $h=h'+h''$.

\medskip

Since $C$ is horizontal, the orthonormal basis $(v_1,v_2)$ is actually  the canonical basis of $\R^2$.
Let $v(\tilde{e}) \in \R^2$ be vector associated to $\tilde{e}$.
By definition, $v(\tilde{e})$ is the period vector of $e$.
A consequence of the fact that $0$ is contained in $H$ is that the  horizontal component of $v(\tilde{e})$ has length at most $\max\{|a'|,|a''|\}$.
Since $\max\{|a'|,|a''|\} \leq \ell$,  (i) follows.

\medskip

It follows from \eqref{ineq:long:cyl:h1} and \eqref{ineq:long:cyl:h2} that
$$
2R-h= (R-h')+(R-h'') < \frac{A}{\alpha^2 h} \Leftrightarrow 2R < h+\frac{A}{\alpha^2 h}.
$$
Since $|e|\leq 2R$, we get that
$$
|e| < h+\frac{A}{\alpha^2 h} < h+\frac{\sqrt{A}}{\alpha^3},
$$
and (ii) follows.

\medskip

Let $C'$ be another long cylinder. By Lemma~\ref{lm:long:cyls:disjoint}, we know that $C$ and $C'$ are disjoint.
Since $e$ crosses both $C$ and $C'$, we would have $|e| > h+\alpha\sqrt{A}$, which implies
$$
\alpha\sqrt{A} < \frac{\sqrt{A}}{\alpha^3}.
$$
Since $\alpha >1$, this impossible.
The same argument shows that $e$ can only cross $C$ once, and (iii) is proved.
\end{proof}

Now, let $(X,q)$ be an element of $\strate$. Fix a homeomorphism $f: (M,\Sig) \ra (X,Z(q))$, and let $(\hat{M},\hat{\Sig})$ and $T_\zeta$ be as in the previous sections.
We endow $M$ and $\hat{M}$ with the flat metrics induced by $q$ via $f$ and $\pi\circ f$. By definition, $T_\zeta$ is an isometry of $(\hat{M},\hat{\Sig})$ of order $d$.
As a direct consequence of the construction of Delaunay triangulation, we get
\begin{Proposition}\label{prop:inv:Delaunay:tria}
There is a Delaunay triangulation of $(\hat{M},\hat{\Sig})$ that is invariant by $T_\zeta$.
 \end{Proposition}
 \begin{proof}
 Since $\Sig$ is invariant under $T_\zeta$, the Voronoi decomposition of $(\hat{M},\hat{\Sig})$ is invariant under $T_\zeta$. In particular the set of Voronoi $0$-cells is invariant under $T_\zeta$. It follows that $T_\zeta$ maps a Delaunay $2$-cell onto a Delaunay $2$-cell. Thus the Delaunay decomposition is invariant under $T_\zeta$.

 Since $T_\zeta$ does not have fixed points in $\hat{M}\setminus\hat{\Sig}$, it acts freely on the set of Delaunay $2$-cells. Pick a representative of each $T_\zeta$-orbit of Delaunay $2$-cells, and subdivide it into triangles, we get a subdivision of all the Delaunay $2$-cells into triangles by applying $T_\zeta$. Thus, we have constructed a Delaunay triangulation invariant by $T_\zeta$.
 \end{proof}

%*******************************************************
%*******************************************************
%*******************************************************
%*******************************************************

\section{Finiteness of the volume of $\Pb\strate$}\label{sec:proof:finite}
Recall that by Corollary~\ref{cor:P:bijection}, we can identify $\strate$ with $\lstrate$.
Let $\ilstrate$ be the subset of $\lstrate$ consisting of elements  $(\hat{X},\hat{\omega},\tau)$ such that the area of the flat surface defined by $\hat{\omega}$ is at most $1$.
By definition, we have
$$
\vol_1(\Pb\strate)=\vol(\ilstrate).
$$
Theorem~\ref{thm:main} will follow from
\begin{Theorem}\label{th:vol:finite:a:diff}
The volume of $\ilstrate$ with respect to $d\vol$ is finite.
\end{Theorem}

To prove Theorem~\ref{th:vol:finite:a:diff} we will use a similar strategy to the one introduced by Masur and Smillie in \cite[Sect. 10]{MS91}. Namely, we will cover $\lstrate$ by  a finite family of open subsets arising from the Delaunay triangulations. We then show that the intersection of each open in this family with the set $\ilstrate$ has  finite volume.

\subsection{Invariant triangulations}\label{sec:encode:tria}
Let $\hat{M}$ be a compact oriented topological surface of genus $\hat{g}$, and $\hat{\Sig}$ a subset of cardinality $\hat{n}$ of $\hat{M}$. We assume that there is homeomorphism $T$ of $\hat{M}$ of order $d$ preserving the set $\hat{\Sig}$  such that the group $\langle T \rangle$ acts freely on $\hat{M}\setminus\hat{\Sig}$.

Let $\hat{\Tcal}$ be a triangulation of $\hat{M}$ whose vertex set is equal to $\hat{\Sig}$. Let $N_1$ and $N_2$ be the number of edges and triangles of $\hat{\Tcal}$ respectively. Note that  $N_1,N_2$ are completely determined by $(\hat{g},\hat{n})$. Namely,
$$
N_1=3(2\hat{g}-2+\hat{n}) \text{ and } N_2=2(2\hat{g}-2+\hat{n}).
$$
We assume further that $\hat{\Tcal}$ is invariant under $T$.

The triangulation $\hat{\Tcal}$ can be encoded by a graph $G$ as follows: the vertex set $V$ of $G$ has cardinality $\hat{n}$, the edges of $G$ are oriented and the set of edges $E$ has cardinality $2N_1$. There is a natural pairing on $E$, each pair consists of two directed edges that correspond to the same edge of $\hat{\Tcal}$ with inverse orientations. We  encode this pairing by a permutation $\sig_0$ on $E$ whose every cycle has length $2$. %We require that the geometric realization of $G$ is homeomorphic to $\hat{\Tcal}$.
At each vertex $v\in V$, we have a cyclic ordering on the set of directed edges pointing out from $v$.  This ordering is induced by the orientation  of $\hat{M}$.  The cyclic orderings at the vertices of $G$ is encoded by a permutation $\sig_1$ on $E$.
In the literature, the graph $G=(V,E)$ together with the permutations $\sig_0,\sig_1$ as above is called a {\em map} (see \cite{LZ}).

The permutations $\sig_0,\sig_1$ induce another permutation $\sig_2$ on $E$ as follows: let $e$ be a directed edge pointing out from $v$. Let $v'$ be the other endpoint of $e$, and $\bar{e}$ be the edge originated from $v'$ that is paired with $e$. Then $\sig_2(e)$ is the edge that precedes $\bar{e}$ in the cyclic ordering at $v'$. The cycles of $\sig_2$ are called the {\em faces} of $G$.
The set of faces of $G$ will be denoted by $F$. %An edge $e$ is said to be incident to a face if this face is the $\sig_2$-orbit of $e$.
Since $\hat{\Tcal}$ is a triangulation, all cycles of $\sig_2$ have length equal to $3$.
Note that the map $G=(V,E,F,\sig_0,\sig_1)$ completely determines $(\hat{M},\hat{\Sig})$ and $\hat{\Tcal}$ up to isomorphism, and $T$ corresponds to  an automorphism of $G$ of order $d$ acting freely on $E$ and $F$.
%Namely, to each face ($3$-cycle) of $\sig_2$, we associate a triangle in the plane. The involution $\sig_0$ provides us with a rule to glue the sides of the triangles in this family. The resulting surface is homeomorphic to $\hat{M}$, and by construction we also get a triangulation of this surface which is isomorphic to $\hat{\Tcal}$.
%In this setting, $T$ corresponds to  an automorphism of $G$ of order $d$ acting freely on $E$ and $F$.

\begin{Definition}\label{def:graph:inv:triang}
Let $\Alp$ denote the set of pairs $(G,T)$, where $G=(V,E,F,\sig_0,\sig_1)$ is a map (in the sense indicated above)  whose geometric realization is homeomorphic to a triangulation of $\hat{M}$ with vertices in $\hat{\Sig}$, and $T$ is an automorphism of order $d$ of $G$ acting freely on $E$ and $F$.
Elements of $\Alp$ will be called {\em invariant triangulations} of $(\hat{M},\hat{\Sig})$.
Since the cardinalities of $V,E$, and $F$ are all fixed, $\Alp$ is clearly a finite set.
\end{Definition}

\subsection{Invariant triangulations and locally homeomorphic maps to $\lstrate$}\label{sec:opens:n:inv:tria}\hfill\\
Given $o=(G,T) \in \Alp$, where $G=(V,E,F,\sig_0,\sig_1)$, there is an associated linear subspace $W_o \subset \C^{2N_1}$ which is defined by the following system of linear equations (we identify $z\in \C^{2N_1}$ with a map $z: E \ra \C$)
$$
(\Scal_o) \quad
\left\{
\begin{array}{ll}
z(\sig_0(e))=-z(e), & \forall e \in E, \\
z(e)+z(e')+z(e'')=0, & \forall (e,e',e'') \text{ face of } G,\\
z(T(e))=\zeta z(e), & \forall e \in E.
\end{array}
\right.
$$
Let $\theta=(e,e',e'') \in F$ be a face of $G$. Define
$$
\begin{array}{cccc}
\Ab_\theta: & \C^E & \ra & \R \\
       &  z & \mapsto & \frac{-\imath}{2}(\ol{z(e)}z(e')-z(e)\ol{z(e')})
\end{array}
$$
and
$$
\Ab(z)=\sum_{\theta\in F} \Ab_\theta(z).
$$
Note that  $\Ab_\theta(z)>0$ if and only if $(z(e),z(e'),z(e''))$ are the sides of a non-degenerate Euclidian triangle in the plane, and the  counter-clockwise orientation on the border of this triangle agrees with the orientation of $e,e',e''$. We denote by $\tilde{U}_o$ the open subset of $W_o$ which is defined by
$$
\tilde{U}_o:=\{z \in W_o, \; \Ab_\theta(z) >0, \text{ for all } \theta\in F\}.
$$
% \begin{equation}\label{eq:triangle:area}
% \frac{-\imath}{2}(\ol{z(e)}z(e')-z(e)\ol{z(e')}) >0.
% \end{equation}
Given $z\in \tilde{U}_o$, we can construct a translation surface from $z$ as follows: to any face $\theta=(e,e',e'')$ of $G$, we have an associated triangle whose sides are $(z(e),z(e'),z(e''))$. Gluing those triangles together, using the identification given by $\sig_0$, gives us an oriented surface $\hat{M}_z$ endowed with a flat metric with conical singularities. Note that the linear parts of the holonomies of this metric structure are always $\id_{\R^2}$. Therefore, $\hat{M}_z$ is actually a translation surface.
By construction, $\hat{M}_z$ naturally carries a triangulation by geodesics whose $1$-skeleton is the geometric realization of $G$.
This triangulation will be denoted by $\hat{\Tcal}_z$.

By the definition of $\tilde{U}_o$, we have $z(T(e))=\zeta z(e)$ for any edge $e\in E$ and $z\in \tilde{U}_o$.
Since $T$ is an automorphism of $G$, if $\theta=(e,e',e'')$ is a face of $G$, then so is  $T(\theta)=(T(e),T(e'),T(e''))$.
Let $\Delta$ and $\Delta'$ be the triangles associated with $\theta$ and $T(\theta)$ respectively.
Then $\Delta'=\zeta\cdot \Delta$, here we view $\zeta$ as an element of $\mathrm{SO}(2,\R) \simeq \S^1$.
This property implies that the action of $T$ on $\hat{\Tcal}_z$ extends to an isometry of the surface $\hat{M}_z$.
Let $(\hat{X}_z,\hat{\omega}_z)$ be the holomorphic $1$-form that defines $\hat{M}_z$. Then $T$ is given by an automorphism $\tau_z$ of $\hat{X}_z$ which satisfies $\tau_z^*\hat{\omega}_z=\zeta\hat{\omega}_z$. Define $X_z:=\hat{X}_z/\langle \tau_z \rangle$. Let $\varpi_z: \hat{X}_z \ra X_z$ be the natural projection. Since $\tau_z^*\hat{\omega}_z^d=\hat{\omega}^d_z$, there is a (meromorphic) $d$-differential $q_z$ on $X_z$ such that $\varpi^*_zq_z=\hat{\omega}^d$.

Define $U_o$ to be the set of $z \in \tilde{U}_o$ such that $(X_z,q_z) \in \strate$.
The correspondence $z \in U_o \mapsto (\hat{X}_z,\hat{\omega}_z,\tau_z)$  defines   a map $\Psi_o: U_o \ra \lstrate$.

\begin{Proposition}\label{prop:triangulation:fin:chart}
The  map $\Psi_o: U_o \ra \lstrate$ is locally homeomorphic.
The union of $\{\Psi_o(U_o), \; o \in \Alp\}$ covers $\lstrate$.
\end{Proposition}
\begin{proof}
%Let $o=(G,T)$ be as above. By definition, we can identify the geometric realization of $G$ with a triangulation $\hat{\Tcal}$ of $\hat{M}$, with vertex set being $\hat{\Sig}$. The action of $T$ on $\hat{\Tcal}$ then extends to a homeomorphism of $\hat{M}$ which sends  triangles of $\hat{\Tcal}$ onto triangles of $\hat{\Tcal}$. In particular, $T$ has order $d$.
%
Let $o=(G,T)$ be as above.
The triangulation $\hat{\Tcal}$ endows $\hat{M}$ with a $\Delta$-complex structure. Thus the system
$$
(\Scal'_o) \qquad \left\{
\begin{array}{ll}
   z(\sig_0(e)) =- z(e) & \text{ for all } e \in E \\
   z(e)+z(e')+z(e'')=0 & \text{ for all } (e,e',e'') \text{ face of } G
 \end{array}
\right.
$$
defines $H^1_{\hat{\Tcal}}(\hat{M},\hat{\Sig},\C)\simeq H^1(\hat{M},\hat{\Sig},\C)$. Consequently, solutions of the system $(\Scal_o)$ represent elements of the space $V_\zeta=\ker(T-\zeta\Id) \subset H^1(\hat{M},\hat{\Sig},\C)$.
Since $\lstrate$ is locally identified with $V_\zeta$, we conclude that the map $\Psi_o$ is locally homeomorphic.

Let $(\hat{X},\hat{\omega},\tau) \in \lstrate$. By Proposition~\ref{prop:inv:Delaunay:tria}, there are  Delaunay triangulations of $(\hat{X},Z(\hat{\omega}))$, where $Z(\hat{\omega})$ is the zero set of $\hat{\omega}$, that are invariant under $\tau$. It follows that $(\hat{X},\hat{\omega},\tau) \in \Psi_o(U_o)$ for some $o\in \Alp$. Hence, $\lstrate$ is covered by $\{\Psi_o(U_o), \; o \in \Alp\}$.
\end{proof}

\subsection{Long cylinders and simple cycles}\label{sec:encode:long:cyl}
Let $o=(G,T)$ be an element of $\Alp$, where $G=(V,E,F,\sig_0,\sig_1)$.
Recall that the geometric realization of $G$ is a triangulation $\hat{\Tcal}$ of the topological surface $\hat{M}$ whose vertex set is equal to $\hat{\Sig}$.

Let $\hat{\Tcal}^\vee$ denote the dual graph of $\hat{\Tcal}$.
%By definition, $\hat{\Tcal}^\vee$ has $N_2$ vertices, $N_1$ edges, and $\hat{n}$ faces.
The automorphism $T$ of $\hTcal$ induces an automorphism of order $d$ on $\hat{\Tcal}^\vee$ that we will abusively denote by $T$.
We will call the image of an injective continuous map from $\S^1$ to $\hat{\Tcal}^\vee$ a {\em simple cycle}  of $\hat{\Tcal}^\vee$.
A simple cycle of $\hat{\Tcal}^\vee$ corresponds to a simple closed curve on $\hat{M}$ which intersects each edge of $\hat{\Tcal}$ at most once.
Since any simple cycle is uniquely determined by the set of edges it contains,  the set of simple cycles in $\hat{\Tcal}^\vee$ is finite.

\begin{Lemma}~\label{lm:Delaunay:long:cyl:finite}
Let $(\hat{X},\hat{\omega},\tau)$ be an element of $\lstrate$, and $\hat{f}: (\hat{M},\hat{\Sig}) \to (\hat{X},\hat{Z}(\hat{\omega}))$ a homeomorphism.
Assume that $\hat{\Tcal}$ is the pullback of a Delaunay triangulation of $\hat{X}$ via $\hat{f}$.
If $c$ is a simple closed curve on $\hat{M}$ such that $\hat{f}(c)$ is a core curve of a long cylinder on $\hat{X}$ (see Def.~\ref{def:long:cyl}),
then $c$ is dual to a simple cycle in $\hat{\Tcal}^\vee$. In particular, the homotopy class of $c$ belongs to a finite family
\end{Lemma}
\begin{proof}
By Lemma~\ref{lm:sc:cross:long:cyl}(iii), each edge of a Delaunay triangulation  of $\hat{X}$ crosses $f(c)$ at most once. Thus each edge of $\hat{\Tcal}$ crosses $c$ at most once.
\end{proof}

%This finiteness is a key ingredient of the proof of Theorem~\ref{th:vol:finite:a:diff}.

Given a simple cycle $\g$ in  $\hat{\Tcal}^\vee$, we denote by $F_\g$ the set of faces of $G$ (triangles of $\hat{\Tcal}$) that are dual to the vertices of $\hat{\Tcal}^\vee$ contained in $\g$.
Let $\hat{\Tcal}^{(1)}_\g$ denote the set of edges of $\hat{\Tcal}$ that are dual to the edges of $\hat{\Tcal}^\vee$ contained in $\g$.
Each edge of $\hat{\Tcal}$ is a pair of edges in $E$ that are permuted by $\sig_0$,  let $E_\g$ denote the set of edges in $E$ corresponding to the (undirected) edges in $\hat{\Tcal}^{(1)}_\g$.  Let $E'_\g$ be the subset of $E$ consisting of edges that are incident to some face in $F_\g$ but not contained in $E_\g$.

Recall that $W_o$ is the subspace of $\C^{2N_1}\simeq \C^{E}$ consisting of solutions to the system $(\Scal_o)$.
As a preparation to the proof of Theorem~\ref{th:vol:finite:a:diff}, we investigate the relations of coordinates $\{z(e), \, e \in E_\g\cup E'_\g\}$ for $z\in W_o$.

\begin{Lemma}\label{lm:sim:cycle:rel:1}
Let $z: E \ra \C$ be a vector in $W_o$.
\begin{itemize}
\item[(a)] Let $e_1,e_2$ be two edges in $E_\g$. Then either $z(e_1)+z(e_2)$ or $z(e_1)-z(e_2)$ is a linear combination of $\{z(e), \, e \in E'_\g\}$.

 \item[(b)] Consider $z$ as an element of $H^1(\hat{M},\hat{\Sig},\C)$, and identify $\g$ with a simple closed curve of $\hat{M}$, then $z(\g)$ is a combination of $\{z(e), \; e \in E'_\g\}$.
\end{itemize}
\end{Lemma}
\begin{proof}
 Glue the triangles in $F_\g$ together using the identification of the edges in $E_\g$, we obtain an annulus whose boundary is composed by edges in $E'_\g$. Remark that any edge in $E_\g$ connects two points in different boundary components of this annulus, hence (a) follows. Since $\g$ can be realized as the core curve of this annulus, (b) follows as well.
\end{proof}

Let $\tilde{\g}:=\{\g_{ij}, \; i=1,\dots,k, j=0,\dots,d-1\}$ be a collection of simple cycles in $\hat{\Tcal}^\vee$ satisfying the following
\begin{itemize}
 \item[(i)] the cycles in $\tilde{\g}$ are pairwise disjoint,

 \item[(ii)] for each $i\in \{1,\dots,k\}, \; \g_{ij}=T^j(\g_{i0}), \; j=0,\dots,d-1$.
\end{itemize}
We will call $\tilde{\g}$ an {\em admissible family} of simple cycles in $\hat{\Tcal}^\vee$.
%Set $E_{ij}=E_{\g_{ij}}$, $E'_{ij}=E'_{\g_{ij}}$, and $F_{ij}=F_{\g_{ij}}$.
Define
$$
E_{\tilde{\g}}:=\bigcup_{\g_{ij}\in \tilde{\g}} E_{\g_{ij}} \qquad \text{ and } \qquad E^*_{\tilde{\g}}:=E\setminus E_{\tilde{\g}}.
$$
For $i=1,\dots,k$, pick an edge $e_i$ in $E_{\g_{i0}}$.
We will call $\{e_1,\dots,e_k\}$ an {\em admissible family of crossing edges} for $\tilde{\g}$.

Recall that $W_o$ can be identified with $\ker(T-\zeta\Id) \subset H^1(\hat{M},\hat{\Sig},\C)$.
Since $\g_{ij}$ is a cycle in $H_1(\hat{M},\hat{\Sig},\Z)$, its defines an element of $W^*_o$.
The value of this linear form at a point $z\in W_o$ will be denoted by $z(\g_{ij})$.

\begin{Lemma}\label{lm:fam:sim:cyc}
There exist $(N-k)$ edges $e_{k+1},\dots,e_N$ in $E^*_{\tilde{\g}}$ such that

\begin{itemize}
 \item[(a)] the map
$$
\begin{array}{cccc}
\Phi: & W_o & \ra & \C^N\\
       & z   & \mapsto & (z(e_1),\dots,z(e_N))
\end{array}
$$
is an isomorphism.

\item[(b)] for all $z\in W_o$ and $\g_{ij}\in \tilde{\g}$, $z(\g_{ij})$ is a linear combination of $(z(e_{k+1}),\dots,z(e_N))$.
\end{itemize}
\end{Lemma}
\begin{proof}
We first show
\begin{Claim}\label{clm:fam:sim:cyc:1}
We have $E'_{\g_{ij}} \subset E^*_{\tilde{\g}}$, for all $\g_{ij}\in \tilde{\g}$.
\end{Claim}
\begin{proof}
Let $e$ be and edge in $E'_{\g_{ij}}$. By definition, $\g_{ij}$ contains a vertex of $\hat{\Tcal}^\vee$ which represents a face $\theta$ of $\hat{\Tcal}$ such that $e$ is a side of $\theta$. Two other sides of $\theta$ are dual to two edges contained in $\g_{ij}$.
Assume that there exists $(i',j')$ such that $e \in E_{\g_{i'j'}}$. Then the dual of  $e$ is contained in $\g_{i'j'}$, hence $\theta$ is contained in $\g_{i'j'}$. But  this implies that, either $\g_{ij}$ passes through $\theta$ twice (in the case $(i',j')=(i,j)$), or $\g_{i'j'}\cap \g_{ij}\neq \varnothing$ (if $(i',j')\neq (i,j)$). Since both situations are excluded by the definition of admissible family of simple cycles, we conclude that $e\in E^*_{\tilde{\g}}$.
\end{proof}

We consider the edges in $E_{\tilde{\g}}$ as elements of $H^1(\hat{M},\hat{\Sig},\C)^*$.
By restricting to $V_\zeta=\ker(T-\zeta\id)\simeq W_o$, we will get elements of $W^*_o$.

\begin{Claim}\label{clm:fam:sim:cyc:2}
The vectors $\{e_1,\dots,e_k\}$ are independent in $W^*_o$.
\end{Claim}
\begin{proof}
For each $i\in \{1,\dots,k\}$, consider the cycle
\begin{equation}\label{eq:def:inv:cycle}
\eta_i:=\sum_{j=0}^{d-1} \zeta^{-j}\g_{ij} \in H_1(\hat{M}\setminus\hat{\Sig},\C)
\end{equation}
By Poincar\'e  duality we can identify $\eta_i$ with  an element of $H^1(\hat{M},\hat{\Sig},\C)$ via the pairing $\langle .,.\rangle : H_1(\hat{M}\setminus\hat{\Sig},\C)\times H_1(\hat{M},\hat{\Sig},\C) \ra \C$. Since $T(\eta_i)=\zeta\eta_i$, we have $\eta_i\in W_o$.
For all $i'\in \{1,\dots,k\}$, we have
$$
\eta_i(e_{i'}):=\langle \eta_i, e_{i'} \rangle =\pm \delta_{ii'}.
$$
Therefore, $\{e_1,\dots,e_{k}\}$ are independent in $W^*_o$.
\end{proof}

\begin{Claim}\label{clm:fam:sim:cyc:3}
We have $W^*_o=\mathrm{Span}(\{e_1,\dots,e_k\}\cup E^*_{\tilde{\g}})$.
\end{Claim}
\begin{proof}
Given $z\in W_o$, we need to show that for any $e\in E_{\tilde{\g}}$, $z(e)$ is a combination of $\{z(e_1),\dots,z(e_k)\}$ and $\{z(e'), \; e' \in E^*_{\tilde{\g}}\}$. Assume that $e \in E_{\g_{ij}}$.
Recall that $e_i\in E_{\g_{i0}}$.
Therefore, $e'_i:=T^j(e_i)$ is an element of $E_{\g_{ij}}$.
Since $z\in W_o=\ker(T-\zeta\id)$, we have $z(e'_i)=\zeta^jz(e_i)$.
By Lemma~\ref{lm:sim:cycle:rel:1}, $z(e)$ is a combination of $z(e'_i)$ and $\{z(e'), \, e'\in E'_{\g_{ij}}\}$.
By Claim~\ref{clm:fam:sim:cyc:1}, we have $E'_{\g_{ij}}$ is contained in $ E^*_{\tilde{\g}}$.
Hence $z(e)$ is a combination of $z(e_i)$ and $\{z(e'), \; e' \in E^*_{\tilde{\g}}\}$ and the claim is proved.
\end{proof}
Claim~\ref{clm:fam:sim:cyc:2} and Claim~\ref{clm:fam:sim:cyc:3} imply that we can complete the family $\{e_1,\dots,e_k\}$  with $(N-k)$ edges $e_{k+1},\dots,e_N$ in $E^*_{\tilde{\g}}$ to obtain a basis of $W^*_o$, thus (a) follows.

\medskip

By (a) there exist $\lambda_{1},\dots,\lambda_N \in \C$, such that for all $z \in W_o$, we have
$$
z(\g_{ij})=\sum_{s=1}^N \lambda_s \cdot z(e_s).
$$
For  $m \in \{1,\dots,k\}$, let $\eta_m$ be the element of $\ker(T-\zeta\Id)$ defined by \eqref{eq:def:inv:cycle}.
For all $\g_{ij}\in \tilde{\g}$, we have $\eta_{m}(\g_{ij})=\langle \eta_{m},\g_{ij}\rangle =0$.
Thus
$$
0=\langle \eta_{m}, \g_{ij}\rangle = \sum_{s=1}^N \lambda_s\cdot\langle \eta_m, e_s\rangle =\pm \lambda_m
$$
since $e_s$ does not intersect $\eta_m$ if $s \neq m$.
Therefore, for all $z\in W_o$,  $z(\g_{ij})$ does not depend on $(z(e_1),\dots,z(e_k))$, and (b) follows.
\end{proof}

\subsection{Finiteness of the volumes of domains associated to admissible families of simple cycles}\label{sec:sim:cycles:fin:vol}
%Let $o=(G=(V,E,F,\sig_0,\sig_1),T)$ be an invariant triangulation in $\Alp$, and $\hat{\Tcal}$ the associated triangulation of $\hat{M}$.
Let $\tilde{\g}:=\{\g_{ij}, \; i=1,\dots,k, j=0,\dots,d-1\}$ be an admissible family of simple cycles in $\hat{\Tcal}^\vee$, and $\{e_1,\dots,e_k\}$ an admissible family of crossing edges for $\tilde{\g}$.
As usual, we will identify $\C^{2N_1}$ with the space of functions $z: E \ra \C$.
Recall that $W_o \in \C^E$ is the space of solutions of $(\Scal_o)$, which is identified with $\ker(T-\zeta\Id) \subset H^1(\hat{M},\hat{\Sig},\C)$.

Let $z$ be a vector in $W_o$.
For $i=1,\dots,k$, let $\ell_i:=|z(\g_{i0})|$.
Let $(x_i,y_i)$ be the real coordinates of $z(e_i)$ in the orthonormal basis $(u_i,v_i)$ of $\R^2$, where $u_i$ is the unit vector in the direction of $z(\g_{i0})$.
Note that
$$
x_i=\frac{\mathrm{Re}\left(z(e_i)\ol{z(\g_{i0})}\right)}{\ell_i} \text{ and } y_i =\frac{\mathrm{Im}\left(z(e_i)\ol{z(\g_{i0})}\right)}{\ell_i},
$$
Recall that $U_o$ is the open domain of $W_o \in \C^E$ defined in Section~\ref{sec:opens:n:inv:tria}.
Let $\alpha \in \R_{>0}$ be some fixed constant.
Let $U_o^1(\tilde{\g},\alpha)$ be the set of $z \in U_o$ such that
$$
\left\{
\begin{array}{l}
\Ab(z) \leq 1,  \\
|z(e)| \leq \sqrt{2}\alpha, \;  \text{for all } \; e \in E^*_{\tilde{\g}}\\
|x_i| \leq \ell_i \text{ and } |y_i| < \frac{2}{\ell_i}, \; i=1,\dots,k,
\end{array}
\right.
$$
where $x_i,y_i, \ell_i$ are defined as above.
The following lemma is key for the proof of Theorem~\ref{th:vol:finite:a:diff}.
\begin{Lemma}\label{lm:vol:fin}
Let $\vol$ be a volume form on $W_o$ which is proportional to the Lebesgue measure. Then $\vol(U^1_o(\tilde{\g},\alpha))$ is finite.
\end{Lemma}
\begin{proof}
By Lemma~\ref{lm:fam:sim:cyc}, we can find $(N-k)$ edges $e_{k+1},\dots,e_{N}$ in $E^*_{\tilde{\g}}$ such that the map
$$
\begin{array}{cccc}
\Phi_o: & W_o & \ra & \C^N\\
      & z & \mapsto & (z(e_1),\dots,z(e_N))
\end{array}
$$
is an isomorphism. Let $\Ucal^1_o(\tilde{\g},\alpha)$ denote the image of $U_o^1(\tilde{\g},\alpha)$ under $\Phi_o$.
Let $\vol_{2N}$ denote the standard Lebesgue measure on $\C^{N}$.
It suffices to show that $\vol_{2N}(\Ucal^1_o(\tilde{\g},\alpha))$ is finite.

By definition, if $(z_1,\dots,z_N) \in \Ucal^1_o(\tilde{\g},\alpha)$, then $|z_i| \leq \sqrt{2}\alpha$,  for $i=k+1,\dots,N$. Set
$$
\Bb^{N-k}(\sqrt{2}\alpha):=\{w=(w_1,\dots,w_{N-k}) \in \C^{N-k}, \; |w_i| \leq \sqrt{2}\alpha, \; i=1,\dots,N-k\}.
$$
Given $w=(w_1,\dots,w_{N-k}) \in \Bb^{N-k}(\sqrt{2}\alpha)$, define
$$
\Ucal^1_o(\tilde{\g},\alpha,w):=\{(z_1,\dots,z_k) \in \C^k, \; (z_1,\dots,z_k,w_1,\dots,w_{N-k}) \in \Ucal^1_o(\tilde{\g},\alpha)\}.
$$
From Lemma~\ref{lm:fam:sim:cyc}, we know that $z(\g_{i0})$ is a linear function of $(w_1,\dots,w_{N-k})$.
Thus $\ell_i$ and the basis $(u_i,v_i)$ depend only on $w$.
Let
$$
R_i=\{z_i\in \C, \, z_i=(x_i,y_i) \text{ in the basis } (u_i,v_i), \text{ where } |x_i| \leq \ell_i, \; |y_i| \leq \frac{2}{\ell_i}\}.
$$
Observe that $R_i$ is a rectangle in $\C$ whose area is equal to $2\ell_i\times\frac{4}{\ell_i}=8$.
By definition, $\Ucal^1_o(\tilde{\g},\alpha,w)$ is contained in $R_1\times\dots\times R_k$.
It follows that $\vol_{2k}(\Ucal^1_o(\tilde{\g},\alpha,w)) < 8^k$, where $\vol_{2k}$ is the Lebesgue measure of $\C^k$.
Since $\Bb^{N-k}(\sqrt{2}\alpha)$ clearly has finite volume in $\C^{N-k}$ (with respect to the Lebesgue measure), the lemma follows from Fubini Theorem.
\end{proof}

\subsection{Proof of Theorem~\ref{th:vol:finite:a:diff}}
\begin{proof}
In what follows, we fix  $\alpha =  2\sqrt{\frac{2}{\pi}}$.
Let $(\hat{X},\hat{\omega},\tau)$ be an element of $\ilstrate$.
Our goal is to show that there is $o\in \Alp$, and an admissible family of simple cycles $\tilde{\g}$ in the dual graph of the triangulation associated to $o$ such that $(\hat{X},\hat{\omega},\tau)\in \Psi(U^1_o(\tilde{\g},\alpha))$, where $U^1_o(\tilde{\g},\alpha)$ is defined as in Section~\ref{sec:sim:cycles:fin:vol}.

%Since $\Alp$ is a finite set, and the set of invariant family of simple cycles in a given triangulation is also finite, Lemma~\ref{lm:Delaunay:long:cyl:finite} allows us to conclude.

By Proposition~\ref{prop:inv:Delaunay:tria}, there is a Delaunay triangulation of $(\hat{X}, \hat{Z}(\hat{\omega}))$ which is invariant under the action of $\tau$.  Pullback this triangulation via a homeomorphism  $\hat{f}: (\hat{M},\hat{\Sig})\ra (\hat{X},\hat{Z}(\hat{\omega}))$,  we get a triangulation  $\hat{\Tcal}$ of $(\hat{M},\hat{\Sig})$. The pullback of $\tau$ by $\hat{f}$ is a homeomorphism $T$ of $\hat{M}$ that preserves $\hat{\Tcal}$.  The triangulation $\hat{\Tcal}$ and the homeomorphism $T$ then provide us with an element $o=(G,T) \in \Alp$.
In what follows, we will identify $\hat{\Tcal}$ with a Delaunay triangulation of $(\hat{X},\hat{Z}(\hat{\omega}))$.

Let $W_o$ and $U_o$ be as in Section~\ref{sec:opens:n:inv:tria}.
Let $z$ denote the vector in $\C^{E}\simeq \C^{2N_1}$ which is defined by $z(e)=\int_{e}\hat{\omega}$.
By definition, we have $z\in U_o$ and $A:=\Aa(\hat{X},\hat{\omega})=\Ab(z)\leq 1$.

\medskip

Recall that a cylinder on $\hat{X}$ is said to be long if its height is greater than $\alpha\sqrt{A}$ (see Def.\ref{def:long:cyl}).
Let $\tilde{C}$ denote the family of all long cylinders in $\hat{X}$.
By  Lemma~\ref{lm:long:cyls:disjoint}, we know that the cylinders in $\tilde{C}$ are pairwise disjoint.
Since $\tau$ is an isometry of $(\hat{X},\hat{\omega})$, the set  $\tilde{C}$ is invariant under $\tau$.
 \begin{Claim}\label{clm:long:cyl:distinct}
 The action of $\tau$ on $\tilde{C}$ is free.
 \end{Claim}
 \begin{proof}
 Let $C$ be a long cylinder and $c$ a core curve of $C$.
 Then $\tau^i(c)$ is a core curve of $\tau^i(C), \; i=0,\dots,d-1$.
 By assumption, we have $z(\tau^i(c))=\zeta^i z(c)$.
 If $\tau^i(C)=C$ then $\tau^i(c)$ must be a core curve of $C$, which means that $\zeta^i\in \{\pm1\}$.
 Since $\zeta^i=1$ only if $i=0 \mod d$, we must have $\zeta^i=-1$.
 This implies that $\tau$ is an isometry of the cylinder $C$ whose derivative is $-\id$.
 In this case, $\tau$ must have two fixed points in the interior of $C$.
 But by definition, a fixed point of $\tau$ must belong to $\hat{\Sig}$, and $C\cap\hat{\Sig}=\varnothing$.
 We then have a contradiction which proves the claim.
 \end{proof}

Claim~\ref{clm:long:cyl:distinct} implies that we can write
 $$
 \tilde{C}:=\{C_{ij}, \; i=1,\dots,k, \; j=0,\dots,d-1\}
 $$
where $C_{ij}=\tau^{j}(C_{i0})$.
By Lemma~\ref{lm:sc:cross:long:cyl}(iii), a core curve $c_{ij}$ of $C_{ij}$ intersects any edge of $\hat{\Tcal}$ at most once, thus it is dual to a simple cycle $\g_{ij}$ in the dual graph $\hat{\Tcal}^\vee$ of $\hat{\Tcal}$.

\begin{Claim}\label{clm:sim:cyc:disjoint}
The simple cycles $\{\g_{ij}, \; i=1,\dots,k, \, j=0,\dots,d-1\}$ are pairwise disjoint.
\end{Claim}
\begin{proof}
 Indeed, if $\g_{ij}$ and $\g_{i'j'}$ share a common vertex (in $\hat{\Tcal}^\vee$), then since every vertex of $\hat{\Tcal}^\vee$ has valency $3$, $\g_{ij}$ and $\g_{i'j'}$ share a common edge, say $e$. This means that $e$ crosses two cylinders with height greater than $\alpha\sqrt{A}$. By Lemma~\ref{lm:sc:cross:long:cyl}(iii),  this is impossible.
\end{proof}

Claim~\ref{clm:sim:cyc:disjoint} implies that $\tilde{\g}:=\{\g_{ij}, \; i=1,\dots,k, \; j=0,\dots,d-1\}$ is an admissible family of simple cycles in $\hat{\Tcal}^\vee$.
Let $E_{\g_{ij}}, E_{\tilde{\g}},E^*_{\tilde{\g}}$, and $(e_1,\dots,e_{k})$ be as in Section~\ref{sec:encode:long:cyl}.

\begin{Claim}\label{clm:short:sc}
For all $e\in E^*_{\tilde{\g}}$, $|z(e)| \leq \sqrt{2}\alpha$.
\end{Claim}
\begin{proof}
If $|z(e)| > \sqrt{2}\alpha$, then by Proposition~\ref{prop:long:sc:high:cyl}, $e$ must cross a cylinder $C$ with height $h$ such that
$|z(e)| < \sqrt{2}h$. It follows that, $h > \alpha$. Hence $C \in \tilde{C}$, and there exist $\g_{ij}\in \tilde{\g}$ such that $e$ is dual to an edge contained in $\g_{ij}$. This  contradicts the hypothesis that $e\in E^*_{\tilde{\g}}$.
\end{proof}

Let $h_i$ and $\ell_{i}$ denote the height and circumference of $C_{i0}$ respectively.
Since $\tau$ is an isometry, the height and circumference of $C_{ij}$ are also  equal to  $h_i$ and $\ell_i$ respectively.

Let $(x_i,y_i)$ be the coordinates of $z(e_i)$ in the orthonormal basis $(u_i,v_i)$ of $\R^2$, where $u_i$ is the unit vector in the direction of $z(\g_{i0})$.
%Note that
%$$
%x_i=\frac{\mathrm{Re}\left(z(e)\ol{z(\g_{i0})}\right)}{\ell_i} \quad \text{ and } \quad y_i =\frac{\mathrm{Im}\left(z(e)\ol{z(\g_{i0})}\right)}{\ell_i}.
%$$
\begin{Claim}\label{clm:long:cyl:area}
We have
\begin{equation}\label{eq:long:cyl:area}
|x_i| \leq \ell_i \quad \text{ and } \quad |y_i| < \frac{2}{\ell_i}.
\end{equation}
\end{Claim}
\begin{proof}
That $|x_i| \leq \ell_i$ follows from Lemma~\ref{lm:sc:cross:long:cyl} (i). To see that $|y_i| \leq \frac{2}{\ell_i}$, we remark that $|y_i| \leq |z(e_i)| < h_i+\frac{\sqrt{A}}{\alpha^3}$ (by Lemma~\ref{lm:sc:cross:long:cyl}(ii)).
Since $h_i >\alpha\sqrt{A} > \frac{\sqrt{A}}{\alpha^3}$, we get $|y_i|<2h_i$.
Now
$$
|y_i|\ell_i < 2h_i\ell_i = 2\Aa(C_{i0}) \leq 2\Aa(\hat{X},\hat{\omega}) \leq 2,
$$
and the claim follows.
\end{proof}
Claim~\ref{clm:short:sc} and Claim~\ref{clm:long:cyl:area} imply that $z\in U^1_o(\tilde{\g},\alpha)$, or equivalently $(\hat{X},\hat{\omega},\tau) \in \Psi(U^1_o(\tilde{\g},\alpha))$.
By Lemma~\ref{lm:vol:fin}, the volume of $\Psi(U^1_o(\tilde{\g},\alpha))$ is finite.
Since $\Alp$ is finite, and $\tilde{\g}$ belongs to a finite set, we deduce that the total volume of $\ilstrate$ is finite.
\end{proof}

\subsection*{Proof of Theorem~\ref{thm:main}}
\begin{proof}
We can now conclude the proof of Theorem~\ref{thm:main}. By  Proposition~\ref{prop:vol:form:def}, $d\vol$ is a well defined volume form on $\strate$ which is parallel with respect to  its affine manifold structure.  Theorem~\ref{th:vol:finite:a:diff} implies that $\vol(\ilstrate)$ is finite. By definition $\vol_1(\Pb\strate)=\vol(\ilstrate)$, hence $\vol_1(\Pb\strate)$ is finite.
\end{proof}

\appendix

\section{Proof of Corollary~\ref{cor:dim:Vzeta}}\label{sec:proof:dim:Vzeta}
\begin{proof}
In what follows, we use the notation of Section~\ref{sec:loc:coord:per}.
If $d=1$, the equations of type \eqref{eq:lin:equa:sym} are all trivial (since $T_\zeta=\id$). Thus the system $(\Scal)$ contains only equations of type \eqref{eq:lin:equa:tria}, hence the space $V$ is isomorphic to $H^1_{\hat{\Tcal}}(\hat{M},\hat{\Sig},\C)=H^1_{\Tcal}(M,\Sig,\C)$. Therefore $N=\dim_\C H^1(M,\Sig,\C)= 2g+n-1$.

Assume from now on that $d>1$. Set $n_1=\card(\Tcal^{(1)})$ and $n_2=\card(\Tcal^{(2)})$. It follows from the Euler characteristic formula for $M$ that
$$
n_1=3(2g+n-2), \quad n_2=2(2g+n-2).
$$
%Since $\Tcal=\hat{\Tcal}/\langle T_\zeta \rangle$, and the action of $T_\zeta$ on $\hat{\Tcal}$ is free, we get $N_1=dn_1, \; N_2=dn_2$.

Let $\hat{E}$ be a subset of $\hat{\Tcal}^{(1)}$ such that each $T_\zeta$-orbit in $\hat{\Tcal}^{(1)}$ has unique representative in $\hat{E}$.
Note that $\card(\hat{E})=n_1$.
Consider elements of $C^{N_1}$ as maps from $\hat{\Tcal}^{(1)}$ to $\C$.
It is clear that if $v \in V_2$, then $v$ is uniquely determined by its restriction to $\hat{E}$.
Thus $\dim V_2={n_1}$.

Let $\hat{F}$ be a subset of $\hat{\Tcal}^{(2)}$ such that each $T_\zeta$-orbit in $\hat{\Tcal}^{(2)}$ has a unique representative in $\hat{F}$.
Let $(\Scal'_1)$ denote the system of linear equations associated to the elements of $\hat{F}$.
We claim that $V$ is equal to the space of solutions to the system $(\Scal'_1)\sqcup (\Scal_2)$.
To see this, let $v$ be vector in $V_2$,  and assume that $v$ is a solution of $(\Scal'_1)$.
For any $\theta \in \Tcal^{(2)}$ there exist $\hat{\theta} \in \hat{F}$, and $k\in \{0,\dots,d-1\}$ such that $\theta=T^k_\zeta(\hat{\theta})$.
Let $\hat{e}_{1},\hat{e}_{2},\hat{e}_{3}$ denote the sides of $\hat{\theta}$, then $T^k_\zeta(\hat{e}_{1}), T^k_\zeta(\hat{e}_{2}),T^k_\zeta(\hat{e}_{3})$ are the sides of $\theta$.
We have $v(T^k_\zeta(\hat{e}_{j}))=\zeta^k v(\hat{e}_{j})$ since $v$ satisfies $(\Scal_2)$, and  $\pm v(\hat{e}_{1})\pm v(\hat{e}_{2})\pm v(\hat{e}_{3})=0$ since  $v$ satisfies $(\Scal'_1)$.
Therefore
$$
\pm v(T^k_\zeta(\hat{e}_{1})) \pm v(T^k_\zeta(\hat{e}_{2})) \pm v(T^k_\zeta(\hat{e}_{3}))=\zeta^k\left(\pm v(\hat{e}_{1})\pm v(\hat{e}_{2})\pm v(\hat{e}_{3})\right)=0
$$
which means that $v$ satisfies all the equations in $(\Scal_1)$. It follows that the space of solutions to $(\Scal'_1)\sqcup(\Scal_2)$ is equal to the space of solutions to $(\Scal_1)\sqcup(\Scal_2)$, that is $V$.
Observe that $(\Scal'_1)$ contains exactly $n_2$ equations. Hence
\begin{equation}\label{eq:dim:V:lower:bd}
\dim V \geq \dim V_2-n_2=n_1-n_2=2g+n-2.
\end{equation}
Consider now the dual graph $\Tcal^\vee$ of $\Tcal$. Let $\Tcal^\vee_0$ be a subgraph of $\Tcal^\vee$ which is a tree, and contains all the vertices of $\Tcal^\vee$. We will call $\Tcal^\vee_0$ a maximal subtree of $\Tcal^\vee$. Let $E_0$ denote the set of edges of $\Tcal$ whose dual is not contained in $\Tcal^\vee_0$. From the Euler characteristic formula, we have $\card(E_0)=2g+n-1$. Let $\Ecal_0 \subset M$ be the union of the edges in $E_0$.

Observe that if we cut $M$ along the edges in $E_0$, then the resulting surface is a topological disc. Thus there is a continuous surjective map
$\varphi: D \ra M$, where $D$ is a closed disc in the plan,   such that the restriction of $\varphi$ to $\inter(D)$ is an embedding, and $\varphi(\partial D)=\Ecal_0$. For each $e\in E_0$, there are two disjoint subintervals $a,a'$ of $\partial D$ such that $\varphi(a)=\varphi(a')=e$, and the restrictions of $\varphi$ to $\inter(a)$ and to $\inter(a')$ are injective.
Therefore, $\partial D$ is decomposed into $2(2g+n-1)$ subintervals together with a pairing defined by the condition: two subintervals are paired if they are mapped to the same edge in $E_0$. Let us denote those subintervals by $a_1,\dots,a_{2(2g+n-1)}$ such that $a_i$ and $a_{2g+n-1+i}$ are paired. We choose the orientation of $a_i$ to be the one induced by the orientation of $D$.

We can lift $\varphi$ to a map $\hat{\varphi}: D \ra \hat{M}$. Note that $\hat{M}_0=\hat{\varphi}(D)$ is a fundamental domain for the action of $\langle T_\zeta \rangle$ on $\hat{M}$. For all $i\in \{1,\dots,2(2g+n-1)\}$, $\hat{\varphi}(a_i)$ is an edge of $\hat{\Tcal}$ which will be denoted by $e_i$. We also endow $e_i$ with the orientation induced by $a_i$.

Let $v$ be a vector in $V$.  For $i=1,\dots, 2g+n-1$, since $e_i$ and $e_{2g+n-1+i}$ project to the same edge of $\Tcal$, there is $r_i\in \{0,\dots,d-1\}$ such that $e_{2g+n-1+i}=-T_\zeta^{r_i}(e_i)$.
Hence $v(e_{2g+n-1+i})=-\zeta^{r_i}v(e_i)$. Since $\{e_1,\dots,e_{2(2g+n-1)}\}$ form the boundary of a disc, we have
\begin{equation}\label{eq:bdry:disc}
 \sum_{i=1}^{2(2g+n-1)}v(e_i)=\sum_{i=1}^{2g+n-1} (1-\zeta^{r_i})v(e_i)=0.
\end{equation}
If $r_i=0$, then $e_{2g+n-1+i}=-e_i$, which means that $e_i$ and $e_{2g+n-1+i}$ are the same edge in $\hat{\Tcal}$ with the inverse orientations.
This happens if and only if $\inter(e_i)$ is contained in the interior of the subsurface $\hat{M}_0$.
If $r_i=0$ for all $i\in \{1,\dots,2g+n-1\}$ then $\hat{M}_0$ is a subsurface of $\hat{M}$ without boundary, which means that $\hat{M}_0=\hat{M}$. But this is impossible in the case $d\geq 2$. Hence there must exist $i \in \{1,\dots,2g+n-1\}$ such that $\zeta^{r_i}\neq 1$. As a consequence, equation \eqref{eq:bdry:disc} is not trivial.

We now claim that $v$ is uniquely determined by $(v(e_1),\dots,v(e_{2g+n-1}))$. Indeed,  any $e \in \hat{\Tcal}^{(1)}$ belongs to the $T_\zeta$-orbit of some edge $\hat{e}$ contained in $\hat{M}_0$. If $\hat{e}$ is contained in $\hat{\varphi}(\partial D)$, then $\hat{e} = e_i$ for some $i\in \{1,\dots,2(2g+n-1)\}$. Otherwise,  $\hat{e}$ is  homologous to a combination of some edges in $\hat{\varphi}(\partial D)$. In all cases, $v(e)=\zeta^k v(\hat{e})$, where $v(\hat{e})$ can be written as a linear function of $(v(e_1),\dots,v(e_{2g+n-1}))$.
The claim implies that the restriction of the map
$$
\begin{array}{cccc}
 \phi: & \C^{N_1} & \ra & \C^{2g+n-1}\\
       & \{v(e), \, e \in \Tcal^{(1)}\} & \mapsto & (v(e_1),\dots,v(e_{2g+n-1}))
\end{array}
$$
to $V$ is injective. Since $\phi(V)$ is contained in the subspace of $\C^{2g+n-1}$ defined by equation~\eqref{eq:bdry:disc}, we get
\begin{equation}\label{eq:dim:V:upper:bd}
\dim V \leq 2g+n-2.
\end{equation}
From \eqref{eq:dim:V:lower:bd} and \eqref{eq:dim:V:upper:bd} we conclude that $\dim V=2g+n-2$.
\end{proof}

%Claim~\ref{clm:Psi:continuous} implies that if $U$ is small enough then $\Psi(U) \subset \widetilde{\Wcal}$.

%***********************************************************
%***********************************************************
%***********************************************************
%***********************************************************

\end{document}